\pdfoutput=1
\documentclass[final,hidelinks,onefignum,onetabnum]{siamart250211}
\usepackage{amsmath,amsfonts,graphicx,placeins}
\usepackage[caption=false]{subfig}
\usepackage{algorithm,algpseudocode}
\usepackage{multirow}
\newcommand{\WeightedMaximum}{1.11\times10^{-14}}
\newcommand{\BetaMaximum}{1.05\times10^{-14}}
\newcommand{\ActionLargestSeconds}{0.108}
\newcommand{\ActionCacheMB}{336.6}
\newcommand{\ActionSlope}{1.02}

\newcommand{\MultiSteps}{6}

\newcommand{\BBOSteps}{5}

\newcommand{\SDCDirectSteps}{42}

\newcommand{\EigenDirectSeconds}{0.0462}
\newcommand{\EigenContourSeconds}{0.0164}

\newcommand{\AbelMinSteps}{6}
\newcommand{\AbelMaxSteps}{20}

\newcommand{\AbelFourBandSeconds}{0.0234}

\newcommand{\AbelFourDirectSeconds}{0.5462}

\newcommand{\AbelHundredBandSeconds}{0.0263}
\newcommand{\AbelHundredDirectSeconds}{0.5251}
\newcommand{\AbelThousandBandSeconds}{0.0283}

\newcommand{\AbelThousandDirectSeconds}{0.5391}

\newcommand{\SweepEightN}{175000}
\newcommand{\SweepEightSteps}{194}

\newcommand{\SweepEightInitialResidual}{1.50\times10^{-4}}

\newcommand{\SweepEightThirdResidual}{5.51\times10^{-12}}

\newcommand{\md}{\mathrm d}
\newcommand{\T}{\top}
\newsiamremark{remark}{Remark}
\headers{CONTOUR ACTIONS FOR JACOBI FRACTIONAL POLYNOMIALS}{X. LIU}
\title{Contour-Based Fast Fractional Integration in a Jacobi Fractional Polynomial Basis}
\author{Xiaolin Liu\thanks{School of Mathematical Sciences, Hainan University, 58 Renmin Avenue, Haikou 570228, Hainan, P.R. China; and Key Laboratory of Engineering Modeling and Statistical Computation of Hainan Province. Email: \email{xiaolin9907@hainanu.edu.cn}.}}
\hypersetup{pdftitle={Contour-Based Fast Fractional Integration in a Jacobi Fractional Polynomial Basis},pdfauthor={Xiaolin Liu}}
\graphicspath{{figures/}}
\begin{document}
\raggedbottom
\maketitle
\begin{abstract}
We develop a contour-based method for applying Riemann--Liouville fractional integrals in a Jacobi fractional polynomial basis with an endpoint weight.
A Gamma-ratio representation reduces the action to contour quadrature and banded ultraspherical solves. For a fixed number of contour nodes, the cost and storage are linear in the polynomial dimension.
For fractional equations, we construct selected diagonals of a short rational approximation directly and use their banded factorizations as reusable right preconditioners for Krylov iterations.
Numerical experiments show accurate integral actions and substantially reduced solution costs across Abel, multi-order, BBO, deferred-correction, eigenvalue, and boundary-value problems.
\end{abstract}
\begin{keywords}
fractional integration, Jacobi fractional polynomials, operator functions, double-exponential quadrature, rational preconditioning, Abel integral equation
\end{keywords}
\begin{MSCcodes}65R20, 65D32, 33C45, 26A33\end{MSCcodes}

\section{Introduction}
We consider the fast application of the Riemann--Liouville fractional integral
\begin{equation}\label{eq:integral}
 (\mathcal I_x^\mu f)(x)=\frac1{\Gamma(\mu)}\int_{-1}^x(x-s)^{\mu-1}f(s)\,\md s,
 \qquad \mu>0,\quad -1<x\le1.
\end{equation}
This weakly singular Volterra operator is a basic building block of
fractional integral equations and integral reformulations of fractional
differential equations. Their solutions often contain noninteger endpoint
powers, motivating spectral bases adapted to these singularities.

The sum-space method of Hale and Olver~\cite{hale18} uses direct sums of
weighted polynomial spaces. For fixed rational orders and smooth
coefficients, it produces banded systems with $O(N)$ construction and
solution costs. However, the combined basis is not orthogonal, and its
finite systems can be severely ill-conditioned~\cite{pu23}. The method is
also restricted to rational orders, with the number of component spaces
increasing with the denominator of the order.

Pu and Fasondini~\cite{pu23} develop a spectral method based on orthogonal
Jacobi fractional polynomials (JFPs), accommodating fractional powers in a
single orthogonal basis. We use the Chebyshev-type form
\begin{equation}\label{eq:basis}
 \varphi_j^{\alpha,\beta}(x)=t^\alpha T_j(2t^\beta-1),\qquad \alpha>-1,\quad\beta>0,\quad j\ge0,
\end{equation}
where $t=(1+x)/2$ and $T_j$ is the Chebyshev polynomial of the first kind.
The endpoint weight and fractional coordinate accommodate algebraic
singularities without restricting the order to rational values. Their
integration-matrix construction, however, requires extended precision for
stabilization and has substantial computational cost.

In our previous work~\cite{liu25}, we develop a stable Chebyshev-based
recurrence that constructs the JFP integration matrix in double precision
with $O(N^2)$ operations and storage. This improves the practicality of the
JFP method, but constructing and applying the complete matrix still has
quadratic cost. The remaining objective is to retain the orthogonal JFP
representation while reducing the integral-action cost to $O(N)$.

The present work achieves this through a Gamma-ratio representation and
contour quadrature~\cite{higham2008,rieder2023}. Each contour node requires
a banded ultraspherical solve~\cite{olver13}. With a fixed number of nodes,
the action uses $O(N)$ operations and storage without forming the complete
integration matrix. A short rational approximation also supplies a reusable
banded right preconditioner for GMRES. Numerical examples assess integral
accuracy and solution costs for fractional equations and eigenvalue problems.

Section~\ref{sec:fast_jfp} derives the action, its quadrature bound, and
its cost model; Sections~\ref{sec:neumann} and
\ref{sec:rational} give the operator bound and banded preconditioner.
Section~\ref{sec:experiments} presents numerical experiments.

\section{Gamma-ratio representation and contour quadrature}
\label{sec:fast_jfp}

\subsection{Jacobi fractional polynomials}
\label{sec:jfp_basis}
The Chebyshev-type JFPs used in~\cite{liu25} are obtained by an algebraic
change of variable followed by an endpoint weight:
\begin{equation}\label{eq:mathcomp_basis}
 t=\frac{1+x}{2},\qquad z=2t^\beta-1,\qquad
 Q_j^{\alpha,\beta}(x)=t^\alpha T_j(z),
 \qquad \alpha>-1,\quad\beta>0.
\end{equation}
They span $t^\alpha$ times polynomials in $t^\beta$. The endpoint exponent
represents a prescribed algebraic singularity, while $\beta$ resolves the
fractional powers left in the regular factor. This is an orthogonal
basis: pulling back the Chebyshev measure gives, up to a positive constant,
\[
 \varpi_{\alpha,\beta}(x)
 =t^{\beta/2-1-2\alpha}(1-t^\beta)^{-1/2},\qquad -1<x<1.
\]
Indeed, $Q_iQ_j\varpi_{\alpha,\beta}(x)\,\md x$ is proportional to
$T_i(z)T_j(z)(1-z^2)^{-1/2}\,\md z$. The basis is therefore evaluated and
manipulated through Chebyshev recurrences. Analyticity of the regular
factor in $z$ yields geometric coefficient decay. For an order
$\mu=p\beta$, $p\in\mathbb N^+$, fractional integration maps this space
into itself with a degree increase of $p$.

\subsection{Exact Gamma-ratio functional calculus}
\label{sec:fast_derivation}
Let $\mu>0$, $\beta>0$, and $\alpha>-1$, and write
$f_N(x)=t^\alpha q(z)$, $t=(1+x)/2$, $z=2t^\beta-1$, $q\in\mathbb P_{N-1}$.
Define the Euler operator and scalar Gamma ratio by
\[
 \mathcal A=\beta(1+z)\partial_z,\qquad
 g_{\mu,\alpha}(\zeta)=
 \frac{\Gamma(1+\alpha+\zeta)}{\Gamma(1+\alpha+\zeta+\mu)}.
\]
The key identity is
\begin{equation}\label{eq:euler_chebyshev_action}
 \mathcal AT_j=\beta j\left(T_0+2\sum_{i=1}^{j-1}T_i+T_j\right),\qquad j\ge1,
 \qquad \mathcal AT_0=0.
\end{equation}
Thus $\mathbb P_{N-1}$ is invariant under $\mathcal A$. Its restriction
to this finite-dimensional polynomial space defines the semigroup and
analytic operator functions used below.

\begin{theorem}[Gamma-ratio representation]\label{thm:gamma_action}
For $\mu>0$, $\beta>0$, $\alpha>-1$, and $q\in\mathbb P_{N-1}$,
\begin{equation}\label{eq:gamma_action}
 (\mathcal I_x^\mu f_N)(x)
 =2^\mu t^{\alpha+\mu}[g_{\mu,\alpha}(\mathcal A)q](2t^\beta-1).
\end{equation}
If $\mu=p\beta$ with $p\in\mathbb N^+$, division by $t^\alpha$
yields a polynomial of degree at most $N+p-1$ in $z$.
\end{theorem}
\begin{proof}
Substituting $s=-1+(1+x)r$ and then $r=e^{-y}$ in \eqref{eq:integral} gives
\begin{align*}
 (\mathcal I_x^\mu f_N)(x)
 &=\frac{2^\mu t^{\alpha+\mu}}{\Gamma(\mu)}\int_0^1
 r^\alpha(1-r)^{\mu-1}q((1+z)r^\beta-1)\,\md r\\
 &=\frac{2^\mu t^{\alpha+\mu}}{\Gamma(\mu)}\int_0^\infty
 e^{-(1+\alpha)y}(1-e^{-y})^{\mu-1}
 q((1+z)e^{-\beta y}-1)\,\md y.
\end{align*}
The dilation $v(y,z)=q((1+z)e^{-\beta y}-1)$ satisfies
$\partial_yv=-\mathcal Av$ and $v(0,z)=q(z)$, hence
$v(y,z)=[e^{-y\mathcal A}q](z)$. Euler's beta integral yields
\begin{equation}\label{eq:gamma_laplace}
\begin{aligned}
 \frac1{\Gamma(\mu)}\int_0^\infty
 e^{-(1+\alpha+\zeta)y}(1-e^{-y})^{\mu-1}\,\md y
 &=\frac{\mathrm B(1+\alpha+\zeta,\mu)}{\Gamma(\mu)}\\
 &=g_{\mu,\alpha}(\zeta),\qquad \Re\zeta>-1-\alpha.
\end{aligned}
\end{equation}
To apply this scalar identity, use the eigenbasis
\[
 \mathcal A(1+z)^k=\beta k(1+z)^k,\qquad 0\le k<N.
\]
For $q(z)=\sum_{k=0}^{N-1}\widehat q_k(1+z)^k$, termwise integration
in the finite sum gives
\[
 \frac1{\Gamma(\mu)}\int_0^\infty
 e^{-(1+\alpha)y}(1-e^{-y})^{\mu-1}e^{-y\mathcal A}q\,\md y
 =\sum_{k=0}^{N-1}\widehat q_k g_{\mu,\alpha}(\beta k)(1+z)^k
 =g_{\mu,\alpha}(\mathcal A)q.
\]
Every integral converges for $\mu>0$ and $\alpha>-1$, proving
\eqref{eq:gamma_action}. The eigenbasis is used only in the proof;
the numerical method does not convert to a power basis.
The operator function preserves degree, and multiplication by
$t^\mu=((1+z)/2)^p$ raises it by at most $p$ when $\mu=p\beta$.
\end{proof}

\begin{remark}[Muti-order]
As in the JFP construction of~\cite{liu25}, several orders
$\mu_r=p_r\beta$, $p_r\in\mathbb N^+$, are treated in a common coordinate
$z=2t^\beta-1$. Theorem~\ref{thm:gamma_action} applies to each order
with the same $\mathcal A_\beta$; only $g_{\mu_r,\alpha}$ and the factor
$2^{\mu_r}((1+z)/2)^{p_r}$ change.
\end{remark}

\subsection{Contour representation and regularization}
For a clockwise closed contour $\mathcal C$ containing the finite
Euler spectrum but none of the Gamma-ratio poles, Cauchy's formula is
\begin{equation}\label{eq:cauchy_functional_calculus}
 g_{\mu,\alpha}(\mathcal A)=\frac1{2\pi i}\oint_{\mathcal C}
 g_{\mu,\alpha}(\zeta)(\mathcal A-\zeta\mathrm{Id})^{-1}\,\md\zeta.
\end{equation}
The following deformation replaces this finite closed path by one
infinite branch and a closing arc whose contribution tends to zero.
The orientation and the constant correction are specified explicitly.

\paragraph{Endpoint exponent and pole separation}
The poles of $g_{\mu,\alpha}$ lie among $-1-\alpha,-2-\alpha,\ldots$;
for integer $\mu=n_\mu$, only the first $n_\mu$ remain. The Euler spectrum
contains the eigenvalue $0$. Thus the
contour must enclose $0$ while its left branch passes between $0$ and the
nearest pole $-1-\alpha$ (Figure~\ref{fig:contour}). The separation is $1+\alpha$ and tends to
zero as $\alpha\to-1^+$. A fixed negative vertex eventually lies on the
wrong side of the approaching pole.
Rather than move the contour into this shrinking gap, we use the Gamma
recurrence
\begin{equation}\label{eq:alpha_shift}
 g_{\mu,\alpha}(\mathcal A)
 =\left[\mathrm{Id}+\mu\bigl(\mathcal A+(1+\alpha)\mathrm{Id}\bigr)^{-1}\right]
 g_{\mu,\alpha+1}(\mathcal A),\qquad -1<\alpha<0.
\end{equation}
This follows from $\Gamma(w+1)=w\Gamma(w)$; both sides are analytic
on the nonnegative spectrum. The contour evaluates $g_{\mu,\alpha+1}$,
whose nearest pole lies strictly to the left of $-1$, and one additional
real shifted solve applies the prefactor. For $\alpha\ge0$ no shift is needed.
This moves the nearest pole away from the zero eigenvalue in the contour
evaluation and handles the approaching pole exactly without increasing
the number of contour nodes.

\begin{figure}[!htbp]\centering
\includegraphics[width=.65\textwidth]{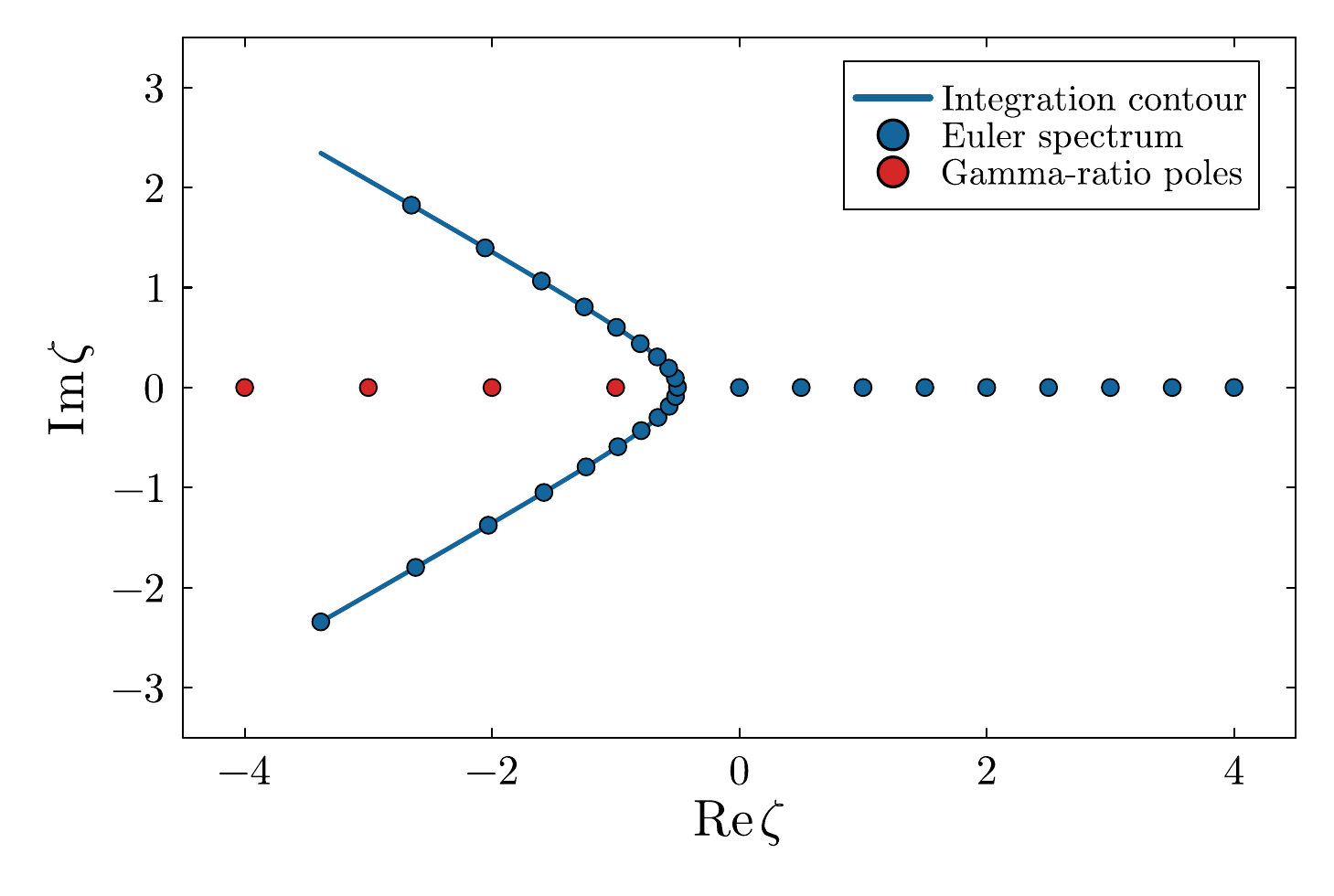}
\caption{Contour geometry for $\alpha=0$, $\beta=\mu=1/2$.
The path separates the nonnegative Euler spectrum from the negative
Gamma-ratio poles. The clockwise closure on the right is sent to infinity.}
\label{fig:contour}
\end{figure}
\FloatBarrier

\paragraph{Contour and constant correction}
For a nonnegative endpoint exponent, we consider the left branch of a
hyperbola with a double-exponential parameterization,
\begin{equation}\label{eq:de_contour}
 \zeta(y)=-\delta\cosh(\kappa\sinh y)
 +i\delta\tan\theta\,\sinh(\kappa\sinh y),
 \qquad y\in\mathbb R.
\end{equation}
The contour geometry is shown in Figure~\ref{fig:contour}.
The parameters $\delta$, $\theta$, and $\kappa$ are fixed at the values
specified in Section~\ref{sec:experiments}, independently of $N$.
The curve is traversed from its lower to its upper branch and closed
clockwise on the right. It encloses the entire nonnegative spectrum,
including zero, and excludes the Gamma-ratio poles $-1-\alpha,-2-\alpha,\ldots$.
When the original exponent is negative, the following formulas are
applied with $\alpha+1$, followed by \eqref{eq:alpha_shift}.
The double-exponential parameterization follows the contour-quadrature
idea of Rieder~\cite{rieder2023}, with the geometry reflected to the left.
The following estimates use the polynomial maximum norm.

We use the following exact resolvent identity:
\begin{equation}\label{eq:resolvent_correction}
 (\mathcal A-\zeta\mathrm{Id})^{-1}+\frac{\mathrm{Id}}{\zeta}
 =\frac{(\mathcal A-\zeta\mathrm{Id})^{-1}\mathcal A}{\zeta}.
\end{equation}
This identity holds for every nonzero $\zeta$ outside the spectrum of
$\mathcal A$ and is used at every contour node.
On the clockwise closed path $\mathcal C$, we have
\begin{equation}\label{eq:clockwise_residues}
\begin{aligned}
 \frac1{2\pi i}\oint_{\mathcal C}
 g_{\mu,\alpha}(\zeta)(\mathcal A-\zeta\mathrm{Id})^{-1}\,\md\zeta
 &=g_{\mu,\alpha}(\mathcal A),\\
 \frac1{2\pi i}\oint_{\mathcal C}
 g_{\mu,\alpha}(\zeta)\frac{\mathrm{Id}}{\zeta}\,\md\zeta
 &=-g_{\mu,\alpha}(0)\mathrm{Id}.
\end{aligned}
\end{equation}
Adding back $g_{\mu,\alpha}(0)\mathrm{Id}$ thus preserves the exact operator.
For real-valued $q$, $\zeta(-y)=\overline{\zeta(y)}$ and
$\zeta'(-y)=-\overline{\zeta'(y)}$. Pairing the two halves gives
\begin{equation}\label{eq:contour_action}
\begin{aligned}
 g_{\mu,\alpha}(\mathcal A)q
 ={}&g_{\mu,\alpha}(0)q+\frac1\pi\int_0^\infty
 \Im\Bigl\{\zeta'(y)g_{\mu,\alpha}(\zeta(y))\\
 &\qquad\times\left[(\mathcal A-\zeta(y)\mathrm{Id})^{-1}
              +\zeta(y)^{-1}\mathrm{Id}\right]q\Bigr\}\,\md y.
\end{aligned}
\end{equation}
The correction in \eqref{eq:resolvent_correction} removes the leading
$-\zeta^{-1}\mathrm{Id}$ term of the resolvent expansion, improving its
decay to $O(|\zeta|^{-2})$ on each fixed polynomial space. For constant
$q$, the corrected integral vanishes and the action is exact.

\subsection{Quadrature and linear-cost resolvent actions}
For the numerical action, we now pass to Chebyshev coefficients. Write
$q(z)=\sum_{j=0}^{N-1}q_jT_j$ and
$\mathbf q=(q_0,\ldots,q_{N-1})^\T$.
The coefficient vector of $\mathcal A q$ is generated directly from
\eqref{eq:euler_chebyshev_action}; no dense representation of the Euler
operator is formed. The shifted equation is expressed in the equivalent
Chebyshev--ultraspherical system below.

For $\alpha\ge0$, take $y_j=jh$, $j=0,\ldots,M-1$, $h=Y/(M-1)$, and write
$\zeta_j=\zeta(y_j)$. Let $R_{-\zeta}^{(N)}\mathbf q$ denote the
Chebyshev coefficient vector of
$(\mathcal A-\zeta\mathrm{Id})^{-1}q$, computed from the
banded system; thus $R_b^{(N)}$ denotes the coefficient matrix of
$(\mathcal A+b\mathrm{Id})^{-1}$. For real $\mathbf q$, define
\begin{equation}\label{eq:contour_discrete}
 \widetilde G_{\mu,\alpha,M}^{(N)}\mathbf q
 =g_{\mu,\alpha}(0)\mathbf q+\sum_{j=0}^{M-1}\Im\left\{
  \omega_j\left[R_{-\zeta_j}^{(N)}\mathbf q
       +\frac{\mathbf q}{\zeta_j}\right]\right\},
\end{equation}
where $\omega_j=h\zeta'(y_j)g_{\mu,\alpha}(\zeta_j)/\pi$ for
$1\le j\le M-2$, and the two endpoint weights are halved.
The $M$ systems count only the upper branch
and its real vertex; conjugate systems are not solved again.
For $-1<\alpha<0$, $\widetilde G_{\mu,\alpha,M}^{(N)}$ is defined
by the coefficient form of \eqref{eq:alpha_shift}, with the contour action
at $\alpha+1$ in place of $g_{\mu,\alpha+1}$; the real resolvent is retained exactly.

The next theorem shows that, for a fixed accuracy, the required number of
contour nodes is independent of $N$.

\begin{theorem}[Dimension-independent quadrature]
\label{thm:uniform_quadrature}
Fix $\mu,\beta>0$, $\alpha>-1$, $0<\theta<\pi/2$, and $\kappa>0$.
Take $0<\delta<1+\alpha$ for $\alpha\ge0$, or $0<\delta<2+\alpha$
when \eqref{eq:alpha_shift} is used for $-1<\alpha<0$.
Let $\widetilde q$ have the coefficients \eqref{eq:contour_discrete}, with
$L=M-1$, $Y=\tfrac12\log L$, and $h=Y/L$.
In exact arithmetic, for every real $q\in\mathbb P_{N-1}$,
\begin{equation}\label{eq:uniform_quadrature_rate}
 \|g_{\mu,\alpha}(\mathcal A)q-\widetilde q\|_\infty
 \le C\exp\!\left(-c\sqrt{\frac{M}{\log M}}\right)\|q\|_\infty,
 \qquad M\ge M_0,
\end{equation}
where $\|q\|_\infty=\max_{-1\le z\le1}|q(z)|$ and $C,c,M_0>0$
depend only on the fixed parameters. Thus convergence is superalgebraic,
and any fixed tolerance is attained with $M$ independent of $N$.
\end{theorem}
\begin{proof}
First let $\alpha\ge0$; operator norms are induced by $\|\cdot\|_\infty$,
and all constants below depend only on the fixed parameters.
The dilation $[e^{-s\mathcal A}q](z)=q((1+z)e^{-\beta s}-1)$ is a
contraction for $s\ge0$. Its Laplace resolvent formula gives
\[
 \|(\mathcal A-\zeta\mathrm{Id})^{-1}\|
 \le\int_0^\infty e^{s\Re\zeta}\,\md s=(-\Re\zeta)^{-1},
 \qquad \Re\zeta<0.
\]
Define the full-line integrand
\[
 \Phi_N(y)=\zeta'(y)g_{\mu,\alpha}(\zeta(y))
 \left[(\mathcal A-\zeta(y)\mathrm{Id})^{-1}
       +\zeta(y)^{-1}\mathrm{Id}\right].
\]
Put $u=\kappa\sinh y$, so $\zeta=-\delta\sec\theta\cosh(u-i\theta)$.
Choose $0<\eta<\min\{\theta,\pi/2-\theta\}$ with
$\delta\sec\theta\cos(\theta-\eta)<1+\alpha$, possible since
$\delta<1+\alpha$. For $u=v+it$, $|t|\le\eta$,
$-\Re\zeta\ge\delta\sec\theta\cos(\theta+\eta)\cosh v>0$.
Also $\Im\zeta=\delta\sec\theta\sinh v\sin(\theta-t)$ vanishes only
at $v=0$, where $\zeta>-(1+\alpha)$, so no Gamma pole is met.
The ends stay a positive angle from the negative real axis; the
asymptotic $g_{\mu,\alpha}(\zeta)\sim\zeta^{-\mu}$ and compactness give
\[
 \begin{aligned}
 |g_{\mu,\alpha}(\zeta)|&\le C(\cosh v)^{-\mu},\qquad
 |\partial_u\zeta|\le C\cosh v,\\
 \|(\mathcal A-\zeta\mathrm{Id})^{-1}+\zeta^{-1}\mathrm{Id}\|
 &\le(-\Re\zeta)^{-1}+|\zeta|^{-1}\le C(\cosh v)^{-1}.
 \end{aligned}
\]
Choose $r>0$ with $\kappa r e^r<\eta$. If $|w-x|\le r/\cosh x$,
$x\in\mathbb R$, then
$|\kappa\sinh w-\kappa\sinh x|\le\kappa e^r\cosh x\,|w-x|<\eta$.
The strip estimates therefore give
$\|\Phi_N(w)\|\le C\cosh x\,e^{-a|\sinh x|}$, $a=\kappa\mu>0$.
Cauchy's estimate on this disk and $s=|\sinh x|$ yield, for integers
$p\ge0$ and a fixed $B>0$,
\[
 \begin{aligned}
 \|\Phi_N^{(p)}(x)\|&\le Cp!r^{-p}\cosh^{p+1}x\,e^{-a|\sinh x|},\\
 \int_{\mathbb R}\|\Phi_N^{(p)}(x)\|\,\md x
 &\le2Cp!r^{-p}\int_0^\infty(1+s)^p e^{-as}\,\md s
 \le CB^p(p!)^2.
 \end{aligned}
\]
Here $(1+s)^p$ is integrated by shifting $s+1$ and using
$\int_0^\infty s^p e^{-as}\,\md s=p!/a^{p+1}$.
These derivative bounds make $\Phi_N$ Schwartz. For its infinite
trapezoidal sum, Poisson summation, $p$ integrations by parts, and
$\sum_{k\ge1}k^{-p}\le2$ give the error $C(Bh)^p(p!)^2$, $p\ge2$.
Taking $p=\lfloor(4Bh)^{-1/2}\rfloor$ and using $p!\le p^p$
bounds this by $C4^{-p}\le Ce^{-c/\sqrt h}$ for small $h$.

Since $b(x)=\cosh x\,e^{-a\sinh x}$ decreases for large $x$,
truncation at $Y$ with half endpoint weights contributes at most
\[
 C\left(hb(Y)+\int_Y^\infty b(x)\,\md x\right)
 =C\left(h\cosh Y+a^{-1}\right)e^{-a\sinh Y}
 \le Ce^{-(a/2)\sinh Y}
\]
for $h\le1$ and large $Y$. Conjugate pairing identifies this truncated
sum with \eqref{eq:contour_discrete}. Thus the total error is bounded by
$C(e^{-c/\sqrt h}+e^{-(a/2)\sinh Y})\|q\|_\infty$.
Since $h^{-1/2}=\sqrt{2L/\log L}$ and $\sinh Y\ge\sqrt L/4$ for
$L\ge2$, \eqref{eq:uniform_quadrature_rate} follows.
For $-1<\alpha<0$, apply the proof at $\alpha+1$ and then
\eqref{eq:alpha_shift}; the prefactor norm is at most
$1+\mu/(1+\alpha)$ by the same resolvent bound, independently of $N$.
\end{proof}

All shifted solves use Chebyshev and ultraspherical coefficients.
Let $S_0$ convert Chebyshev coefficients to $C^{(1)}$ coefficients,
and let $\frac{\beta}{2}D_A$ represent the Euler map in these coordinates.
The identities
$T_0=C_0^{(1)}$, $T_1=C_1^{(1)}/2$, and
$T_j=(C_j^{(1)}-C_{j-2}^{(1)})/2$ for $j\ge2$ give $S_0$.
Combining $T_j'=jC_{j-1}^{(1)}$ with multiplication by $1+z$ gives
$D_A/2$. Their matrix forms are
\begin{equation}\label{eq:jfp_column_maps}
 {\setlength{\arraycolsep}{3pt}
 D_A=\begin{pmatrix}
 0&2&2&&&\\
  &1&4&3&&\\
  &&2&6&4&\\
  &&&\ddots&\ddots&\ddots
 \end{pmatrix},\qquad
 S_0=\frac12\begin{pmatrix}
 2&0&-1&&&\\
  &1&0&-1&&\\
  &&1&0&-1&\\
  &&&\ddots&\ddots&\ddots
 \end{pmatrix}.
 }
\end{equation}
Blank entries are zero. Shifted solves use the leading $N\times N$
sections of these maps. Since the diagonal of $S_0$ is nonzero,
\begin{equation}\label{eq:shifted_system_equivalence}
 (\mathcal A-\zeta\mathrm{Id})v=q
 \quad\Longleftrightarrow\quad
 \left(\frac{\beta}{2}D_A-\zeta S_0\right)\mathbf v=S_0\mathbf q.
\end{equation}
Here $\mathbf v$ contains Chebyshev coefficients of $v$; the equation
is expressed in ultraspherical coordinates. The matrix is upper triangular
with three diagonals and nonsingular on the contour, so forming
$S_0\mathbf q$ and solving cost $O(N)$ operations.
For $|\zeta|\ge\max\{32,2\beta(N-1)^2\}$, we instead solve
$\left(\frac{\beta}{2}D_A-\zeta S_0\right)\mathbf v_{\rm corr}
=\frac{\beta}{2\zeta}D_A\mathbf q$, evaluating the right side of
\eqref{eq:resolvent_correction} without subtracting nearly equal terms.
Below this threshold we use the left side. Compensated summation
accumulates the contour contributions.
For $\mathbf v_+=\widetilde G_{\mu,\alpha+1,M}^{(N)}\mathbf q$,
the extra solve in \eqref{eq:alpha_shift} uses
$[\frac{\beta}{2}D_A+(1+\alpha)S_0]\mathbf v_b=S_0\mathbf v_+$, followed by
$\mathbf v_+\leftarrow\mathbf v_++\mu\mathbf v_b$. It has the same three-band
structure, costs $O(N)$ operations, and reuses cached real coefficients.

The multiplication matrix in Chebyshev coefficients has the following
form used in the ultraspherical spectral method~\cite{olver13}.
\begin{lemma}[Multiplication matrix]\label{thm:M}
Let $\chi(z)=\sum_{j\ge0}\chi_jT_j(z)$ and
$q(z)=\sum_{j\ge0}q_jT_j(z)$ have absolutely summable coefficients.
With $\mathbf q=(q_j)_{j\ge0}$, the coefficient vector of $\chi(z)q(z)$ is $M_{\chi}^{(0)}\mathbf q$, where
\begin{equation}\label{eq:chebyshev_multiplication}
{\setlength{\arraycolsep}{3pt}
 M_{\chi}^{(0)}=\frac12\left[
 \begin{pmatrix}
 2\chi_0&\chi_1&\chi_2&\chi_3&\cdots\\
 \chi_1&2\chi_0&\chi_1&\chi_2&\ddots\\
 \chi_2&\chi_1&2\chi_0&\chi_1&\ddots\\
 \chi_3&\chi_2&\chi_1&2\chi_0&\ddots\\
 \vdots&\ddots&\ddots&\ddots&\ddots
 \end{pmatrix}
 +
 \begin{pmatrix}
 0&0&0&0&\cdots\\
 \chi_1&\chi_2&\chi_3&\chi_4&\ddots\\
 \chi_2&\chi_3&\chi_4&\chi_5&\ddots\\
 \chi_3&\chi_4&\chi_5&\chi_6&\ddots\\
 \vdots&\ddots&\ddots&\ddots&\ddots
 \end{pmatrix}
 \right].
}
\end{equation}
The same matrix applies to multiplication in the unweighted JFP basis,
since $Q_j^{0,\beta}(x)=T_j(z)$.
\end{lemma}
\begin{proof}
Apply $T_iT_j=(T_{i+j}+T_{|i-j|})/2$ and collect coefficients.
\end{proof}
For $\mu=p\beta$, multiplication by the degree-$p$ polynomial
$((1+z)/2)^p$ produces all $N+p$ output coefficients.
The potentially singular weight is never sampled during an action.

\paragraph{Stable scalar weights}
For $0<\mu<1$, put $\eta=1+\alpha+\zeta$.
For $|\eta|<32$, the code evaluates
$\log\Gamma(\eta)-\log\Gamma(\eta+\mu)$ directly. When $\Re \eta<0$ and
$|\eta|\ge32$, the reflection formula gives
\begin{equation}\label{eq:gamma_reflection}
 \frac{\Gamma(\eta)}{\Gamma(\eta+\mu)}
 =\frac{\Gamma(1-\eta-\mu)}{\Gamma(1-\eta)}
   \frac{\sin\pi(\eta+\mu)}{\sin\pi \eta}.
\end{equation}
On the upper branch, the logarithm of the sine ratio is evaluated as
\[
 -i\pi\mu+\operatorname{log1p}(-e^{2\pi i(\eta+\mu)})
          -\operatorname{log1p}(-e^{2\pi i \eta}).
\]
Both exponentials decay with $\Im \eta$, avoiding the overflow of the
individual sines. The reflected Gamma arguments lie in the right
half-plane for $0<\mu<1$.

For $\Re \eta\ge0$ and $|\eta|\ge32$, we retain twelve terms of the Taylor
expansion of $\log\Gamma(\eta)-\log\Gamma(\eta+\mu)$ in $\mu$.
For $0<\mu<1$, this gives sufficient scalar truncation accuracy for
double-precision evaluation, independently of $N$. For a negative endpoint
exponent the weights use $\alpha+1$, as required by
\eqref{eq:alpha_shift}, and are cached for repeated actions.

For stable scalar evaluation at higher orders, write $\mu=n_\mu+\mu_{\rm f}$,
where $n_\mu$ is a nonnegative integer and $0\le\mu_{\rm f}<1$. Then
\[
 \frac{\Gamma(\eta)}{\Gamma(\eta+\mu)}
 =\frac{\Gamma(\eta)}{\Gamma(\eta+\mu_{\rm f})}
   \prod_{j=0}^{n_\mu-1}(\eta+\mu_{\rm f}+j)^{-1}.
\]
The first ratio is one when $\mu_{\rm f}=0$; otherwise it uses the
log-domain evaluation above.


Algorithm~\ref{alg:contour} summarizes the contour preparation and action.

\begin{algorithm}[tb]
\caption{Reusable contour preparation and fractional-integral action}
\label{alg:contour}
\begin{algorithmic}[1]
\Require $N$, order $\mu=p\beta$ with integer $p\ge1$, endpoint exponent
$\alpha>-1$, and $M$ contour nodes.
\Ensure A reusable plan and the integral coefficients.
\State Compute nodes and log-domain scalar weights, using $\alpha+1$
if $\alpha<0$. \hfill $O(M)$
\State Cache inverse diagonals and real row coefficients; allocate vectors.
\hfill $O(MN)$
\State For input $\mathbf q$, form $S_0\mathbf q$, $\frac{\beta}{2}D_A\mathbf q$ and
the constant-mode accumulator. \hfill $O(N)$
\State At each node solve the corrected upper-three-band system and
accumulate its compensated contribution. \hfill $O(MN)$
\State If $\alpha<0$, apply the exact real-resolvent correction.
\hfill $O(N)$
\State Multiply by $2^\mu((1+z)/2)^p$ in coefficient space. \hfill $O(p(N+p))$
\State Return all $N+p$ coefficients.
\hfill $O(N+p)$
\end{algorithmic}
\end{algorithm}
\begin{remark}[Reuse and action cost]
The nodes, scalar weights, and shifted-system coefficients in
Algorithm~\ref{alg:contour} are prepared once and reused, including for
the real and imaginary parts of complex inputs. For fixed $p$,
preparation and each action cost $O(MN)$ operations; the tables require
$O(MN)$ storage and the working vectors $O(N)$. The polynomial multiplication costs
$O(p(N+p))$; fixed $M,p$ give linear costs in $N$.
\end{remark}

\section{Neumann decay and operator GMRES}
\label{sec:neumann}

The second-kind Abel equation is
\begin{equation}\label{eq:abel}
 u(x)+\lambda(\mathcal I_x^\mu u)(x)=f(x).
\end{equation}
We analyze exact GMRES~\cite{saad1986} in
$L^2(-1,1;\mathbb C)$ using the Volterra structure.

\begin{lemma}\label{lem:neumann_candidate}
Let $\mathcal B=\mathrm{Id}+\mathcal K$ be bounded and invertible on a Hilbert space.
Let $r_{\rm res}^{(k)}$ denote its residual after $k$ GMRES
steps. Then
\begin{equation}\label{eq:neumann_candidate}
 \frac{\|r_{\rm res}^{(k)}\|}{\|r_{\rm res}^{(0)}\|}\le\|\mathcal K^k\|,
 \qquad r_{\rm res}^{(0)}\ne0.
\end{equation}
\end{lemma}
\begin{proof}
The minimal-residual characterization is
\[
 \|r_{\rm res}^{(k)}\|
 =\min_{\pi_{\rm res}\in\mathbb P_k,\,\pi_{\rm res}(0)=1}
   \|\pi_{\rm res}(\mathcal B)r_{\rm res}^{(0)}\|.
\]
Choosing $\pi_{\rm res}(z)=(1-z)^k$ gives
$\pi_{\rm res}(\mathcal B)=(-\mathcal K)^k$ and proves the bound.
\end{proof}

The semigroup property gives an explicit bound for the Abel equation.

\begin{theorem}[Abel residual bound]\label{thm:abel_gmres}
For $\mu>0$ and $\lambda\in\mathbb C$, the operator
$\mathrm{Id}+\lambda\mathcal I_x^\mu$ is invertible on $L^2(-1,1)$.
Its exact GMRES residual satisfies
\begin{equation}\label{eq:abel_gamma_bound}
 \frac{\|r_{\rm res}^{(k)}\|}{\|r_{\rm res}^{(0)}\|}
 \le\frac{(|\lambda|2^\mu)^k}{\Gamma(1+k\mu)}.
\end{equation}
For fixed $\mu,\lambda$, this bound decays faster than every
geometric sequence. In particular, for each $0<\gamma_{\rm geom}<1$,
\begin{equation}\label{eq:abel_geometric_bound}
 \frac{\|r_{\rm res}^{(k)}\|}{\|r_{\rm res}^{(0)}\|}
 \le E_\mu(|\lambda|2^\mu/\gamma_{\rm geom})\gamma_{\rm geom}^k,
 \qquad E_\mu(z)=\sum_{j=0}^\infty\frac{z^j}{\Gamma(1+j\mu)}.
\end{equation}
\end{theorem}
\begin{proof}
Young's convolution inequality and the beta integral give, respectively,
\[
 \|\mathcal I_x^\nu\|\le\frac{2^\nu}{\Gamma(1+\nu)},\qquad
 \mathcal I_x^a\mathcal I_x^b=\mathcal I_x^{a+b}
 \quad(\nu,a,b>0).
\]
For the second identity, the convolution of the two kernels is
$\Gamma(a+b)^{-1}(x-s)^{a+b-1}$; extension from bounded inputs
to $L^2$ follows from the first estimate. Hence, with
$c_\lambda=|\lambda|2^\mu$,
$\|(\lambda\mathcal I_x^\mu)^k\|\le c_\lambda^k/\Gamma(1+k\mu)$.
For $c_\lambda>0$, Stirling's formula yields
\begin{equation}\label{eq:neumann_stirling}
 \log\frac{c_\lambda^k}{\Gamma(1+k\mu)}
 =-\mu k\log k
   +k(\log c_\lambda+\mu-\mu\log\mu)+O(\log k).
\end{equation}
The majorant is summable, so the Neumann series converges in
operator norm to the inverse. Lemma~\ref{lem:neumann_candidate}
then gives \eqref{eq:abel_gamma_bound}. Its $k$th root tends to zero.
Each term $(c_\lambda/\gamma_{\rm geom})^k/\Gamma(1+k\mu)$ is bounded by the convergent
positive series $E_\mu(c_\lambda/\gamma_{\rm geom})$, proving
\eqref{eq:abel_geometric_bound}. If $\lambda=0$, GMRES terminates
after one step.
\end{proof}

The same decay extends to finite sums with bounded variable coefficients.

\begin{theorem}[Finite sums and bounded coefficients]\label{thm:multi_gmres}
Let
\[
 \mathcal K f=\sum_{j=1}^s a_{{\rm L},j}(x)\mathcal I_x^{\mu_j}[a_{{\rm R},j} f](x),
 \qquad a_{{\rm L},j},a_{{\rm R},j}\in L^\infty(-1,1),\quad \mu_j>0,\quad s<\infty.
\]
Set
\begin{equation}\label{eq:multi_constant}
 \mu_* =\min_{1\le j\le s}\mu_j,\qquad
 c_{\mathcal K}=\sum_{j=1}^s\|a_{{\rm L},j}\|_\infty\|a_{{\rm R},j}\|_\infty
       2^{\mu_j}\frac{\Gamma(\mu_*)}{\Gamma(\mu_j)}.
\end{equation}
Then $\mathrm{Id}+\mathcal K$ is invertible and its exact
GMRES residual satisfies
\begin{equation}\label{eq:multi_gamma_bound}
 \frac{\|r_{\rm res}^{(k)}\|}{\|r_{\rm res}^{(0)}\|}
 \le\frac{c_{\mathcal K}^k}{\Gamma(1+k\mu_*)}.
\end{equation}
\end{theorem}
\begin{proof}
For $0<\upsilon\le2$, dominate each kernel by
$\upsilon^{\mu_j-1}\le2^{\mu_j-\mu_*}\upsilon^{\mu_*-1}$.
With $c=c_{\mathcal K}/2^{\mu_*}$ this implies
$|\mathcal K f|\le c\mathcal I_x^{\mu_*}|f|$ almost everywhere.
Positivity and induction give
$|\mathcal K^k f|\le c^k\mathcal I_x^{k\mu_*}|f|$.
Young's inequality proves
$\|\mathcal K^k\|\le c_{\mathcal K}^k/\Gamma(1+k\mu_*)$.
Summability, invertibility and the GMRES bound now follow as in
Theorem~\ref{thm:abel_gmres}. No commutation of the variable
coefficients with the integrals is needed.
\end{proof}

Thus the second-kind operators considered above are invertible, with
Neumann series converging in operator norm, and exact GMRES
converges supergeometrically in $L^2(-1,1)$. For fixed equation parameters,
the iteration bound for a prescribed relative residual tolerance is
uniform over all nonzero initial residuals. Its dependence on the
fractional orders and coefficient magnitudes motivates the preconditioning
developed next.


\section{Banded rational preconditioning}
\label{sec:rational}
For the single-order Abel equation \eqref{eq:abel}, take $\beta=\mu$.
The coefficient system is
\begin{equation}\label{eq:abel_finite}
 (\mathrm I_N+\lambda\widetilde I_{N,M}^\mu)\mathbf u=\mathbf f.
\end{equation}
Here $\widetilde I_{N,M}^\mu$ denotes the contour integration matrix.
We use a short rational approximation only to construct a banded right
preconditioner. Equation actions and residual checks retain the accurate
contour rule. GMRES uses fixed right preconditioning and a recomputed original residual.

\subsection{Short rational approximation}
\label{sec:abel_rational}
For $0<\mu<1$ and $\alpha\ge0$, the Gamma-ratio expansion~\cite{olv10} gives
\[
 g_{\mu,\alpha}(\zeta)
 =\left(\zeta+\alpha+\frac{1+\mu}{2}\right)^{-\mu}
 [1+O(|\zeta|^{-2})]
\]
in closed sectors avoiding the negative real axis. For $\Re \eta>0$,
\begin{equation}\label{eq:rational_stieltjes}
 \eta^{-\mu}=\frac{2\sin(\pi\mu)}{\pi}
 \int_{-\infty}^{\infty}\frac{e^{2(1-\mu)s}}{\eta+e^{2s}}\,\md s.
\end{equation}
Take sixteen trapezoidal nodes $s_j=-5+14j/15$, $0\le j<16$,
with endpoint factors $\vartheta_0=\vartheta_{15}=1/2$ and
$\vartheta_j=1$ otherwise.
Since the shifted-power model is already a coarse approximation, we use
$[-5,9]$ to cover the broad range of shifts $e^{2s}\in[e^{-10},e^{18}]$
that suffices for preconditioning in the reported tests, without requiring
high-accuracy evaluation of the Stieltjes integral.
The trapezoidal approximation is
\begin{equation}\label{eq:abel_rational_nodes}
\begin{aligned}
 &\frac{2\sin(\pi\mu)}{\pi}
 \int_{-\infty}^{\infty}
 \frac{e^{2(1-\mu)s}}
 {\zeta+\alpha+(1+\mu)/2+e^{2s}}\,\md s\\
 &\qquad\approx\frac{28\sin(\pi\mu)}{15\pi}
 \sum_{j=0}^{15}
 \frac{\vartheta_j e^{2(1-\mu)s_j}}
 {\zeta+\alpha+(1+\mu)/2+e^{2s_j}}.
\end{aligned}
\end{equation}
We define $r_{\mu,\alpha}$ by rescaling this sum so that
$r_{\mu,\alpha}(0)=g_{\mu,\alpha}(0)$, which matches the exact
constant-mode action without fitting a solution.
For $-1<\alpha<0$, construct $r_{\mu,\alpha+1}$ from the same
quadrature \eqref{eq:abel_rational_nodes} with $\alpha$ replaced by
$\alpha+1$, rescaling the sum to satisfy
$r_{\mu,\alpha+1}(0)=g_{\mu,\alpha+1}(0)$. Then set
\begin{equation}\label{eq:rational_negative_alpha}
 r_{\mu,\alpha}(\zeta)=
 \left(1+\frac{\mu}{\zeta+1+\alpha}\right)
 r_{\mu,\alpha+1}(\zeta).
\end{equation}
Partial fractions give 17 terms for $-1<\alpha<0$ and 16 for
$\alpha\ge0$, with poles and residues independent of $N$, $\lambda$,
and the right-hand side. This rule is used only for preconditioning:
the shifted-power asymptotic and truncated quadrature do not provide a
global error bound for $r_{\mu,\alpha}$. In particular, a fixed finite
sum decays as $\zeta^{-1}$ rather than $\zeta^{-\mu}$, so it cannot be
a uniform relative approximation on the positive half-line.

\subsection{Direct construction of selected diagonals}
Each shifted resolvent can be applied through a banded ultraspherical
solve, so applying the fixed 16- or 17-term rational approximation costs $O(N)$.
Its coefficient matrix is nevertheless generally dense and upper
triangular. This inexpensive action does not directly provide the inverse
of the approximate equation operator, which also includes the coordinate
multiplier and the identity term. An inner iterative solve would add
work at every outer GMRES step. We therefore retain a band of the
approximate equation matrix to obtain $P_N$ in
\eqref{eq:abel_rational_pc}. For fixed bandwidths, one banded LU
factorization and each subsequent application of $P_N^{-1}$ cost $O(N)$.

For each positive shift $b$, write $R_b=R_b^{(N)}$ for the coefficient
matrix defined above. We construct the main diagonal and $K$
superdiagonals of the resolvent coefficient matrix
\begin{equation}\label{eq:band_resolvent_matrix}
 R_b=\left(\frac{\beta}{2}D_A+bS_0\right)^{-1}S_0.
\end{equation}
The right factor $S_0$ converts the input Chebyshev coefficients to the
ultraspherical basis. The coefficient maps are those in
\eqref{eq:jfp_column_maps}, and
\begin{equation}\label{eq:band_resolvent_identity}
 \left(\frac{\beta}{2}D_A+bS_0\right)R_b=S_0.
\end{equation}
The matrix on the left is upper triangular with three diagonals.
For column $l$, back substitution starts at row $k=l$ and stops at
$k=\max\{0,l-K\}$. Thus only the required near-diagonal entries of
$R_b$ are computed.

For stable evaluation, write $\widehat r_{kl}=b(R_b)_{kl}$. With zero-based indices,
the scaled matrix entries are computed for $k=l,l-1,\ldots,\max\{0,l-K\}$ by
\begin{equation}\label{eq:band_recurrence}
 \widehat r_{kl}=
 \frac{(S_0)_{kl}-\beta(k+1)\widehat r_{k+1,l}/b
       -[\beta(k+2)/(2b)-1/2]\widehat r_{k+2,l}}
      {\beta k/(2b)+(S_0)_{kk}}.
\end{equation}
Entries below the main diagonal are zero. Scaling avoids forming very
small shifted coefficients before multiplication by their residues.
Let $\operatorname{diag}_K(X)$ retain the entries $X_{kl}$ with
$0\le l-k\le K$ and set all other entries to zero. Denote the Chebyshev
coefficient matrix of $r_{\mu,\alpha}(\mathcal A_\beta)$ by
$R_{\mu,\alpha}^{(N)}$. For $\alpha\ge0$, its selected entries satisfy
\begin{equation}\label{eq:rational_selected_diagonals}
 \operatorname{diag}_K\bigl(R_{\mu,\alpha}^{(N)}\bigr)
 =\sum_{j=0}^{15}a_j\operatorname{diag}_K(R_{b_j}).
\end{equation}
Here $-b_j$ and $a_j$ are the poles and residues of the
normalized sum in \eqref{eq:abel_rational_nodes}. The scaled entries
are accumulated with weights $a_j/b_j$.
For $-1<\alpha<0$, start with the 16-term rule for
$r_{\mu,\alpha+1}$. Partial fractions in
\eqref{eq:rational_negative_alpha} modify its 16 residues and add one
term with shift $b=1+\alpha$. Its selected entries are computed by the
same recurrence and included in the weighted sum, giving 17 terms.
Let $\operatorname{band}_{\ell,b_{\rm up}}(X)$ be the $N\times N$ matrix with
entries $X_{kl}$ for $-b_{\rm up}\le k-l\le\ell$ and zeros otherwise.
The Abel preconditioner is
\begin{equation}\label{eq:abel_rational_pc}
 P_N=\mathrm I_N+2^\mu\lambda\operatorname{band}_{1,b_{\rm up}}\left(
 M_{(1+z)/2}^{(0)}
 R_{\mu,\alpha}^{(N)}\right).
\end{equation}
Its lower bandwidth is at most one: a function of the upper triangular
Euler operator preserves degree, and the coordinate multiplier raises it
by one. Its upper part is generally dense; that truncation is used only
in the preconditioner. In the Abel implementation, the degree-one factor
$M_{(1+z)/2}^{(0)}$ means that $R_{\mu,\alpha}^{(N)}$ is prepared through
$K=b_{\rm up}+1$ selected superdiagonals, after which the product is written directly
into the final $(1,b_{\rm up})$ band.

More generally, multiplication by a polynomial of degree $d$ couples
only coefficients at most $d$ indices apart. Thus, for a short rational
function $r$ with $m$ terms, let $R_{\rm rat}$ be its coefficient matrix.
For polynomial multiplication matrices $M_{\rm L},M_{\rm R}$ of degrees
$d_{\rm L},d_{\rm R}$, the retained upper bandwidth $b_{\rm up}$ of
$M_{\rm L}R_{\rm rat}M_{\rm R}$ requires only the main diagonal and
$K=b_{\rm up}+d_{\rm L}+d_{\rm R}$ superdiagonals of $R_{\rm rat}$.
The lower bandwidth is at most $d_{\rm L}+d_{\rm R}$. These entries are constructed
directly in $O(m(K+1)N)$ work and $O((K+1)N)$ storage, both linear in $N$
for fixed term count, degrees, and bandwidth. Here and in the cost
estimates below, $m$ records the fixed number of terms: 16 or 17 for
the Abel rule and 32 for the Airy rule introduced later. Any additional
exact resolvent factor is applied separately.

\begin{algorithm}[tb]
\caption{Prepare a banded rational preconditioner}\label{alg:band_prepare}
\begin{algorithmic}[1]
\Require $N$, fractional and basis parameters, $m$ rational terms,
upper bandwidth $b_{\rm up}$, fixed-degree
polynomial multipliers, and lower bandwidth $\ell$.
\Ensure Reusable band LU factors, pivots, and solve workspace.
\State Compute the real poles and residues of the selected short rule.
\hfill $O(m)$
\State Set $K=b_{\rm up}+d_{\rm L}+d_{\rm R}$; allocate retained diagonals. \hfill $O((K+1)N)$
\State Accumulate $K+1$ diagonals via \eqref{eq:band_recurrence}
over all poles. \hfill $O(m(K+1)N)$
\State Apply any fixed number of exact real-resolvent factors to the
diagonals by the same recurrence. \hfill $O((K+1)N)$
\State Apply fixed-degree multipliers and insert retained entries and
identity terms into band storage. \hfill $O((K+1)N)$
\State Factor using LAPACK band LU with pivots in the lower band.
\hfill $O(N\ell(\ell+b_{\rm up}+1))$
\State Store factors and pivots; allocate solve workspace.
\hfill $O(N(2\ell+b_{\rm up}+1))$
\end{algorithmic}
\end{algorithm}
Band elimination increases the upper bandwidth to at most $\ell+b_{\rm up}$;
the LU storage array has $(2\ell+b_{\rm up}+1)N$ entries.

Algorithm~\ref{alg:gmres} applies GMRES with fixed right preconditioning to the
system $B_N\mathbf c=\mathbf f_{\rm sys}$, using the accurate contour
action and recomputing the original residual. Here $\mathbf c$ denotes
the algebraic unknown; for Abel, $\mathbf c=\mathbf u$ and
$\mathbf f_{\rm sys}=\mathbf f$.

\begin{algorithm}[tb]
\caption{Contour GMRES with recomputed-residual correction}\label{alg:gmres}
\begin{algorithmic}[1]
\Require Accurate action $B_N$, prepared preconditioner, right-hand side
$\mathbf f_{\rm sys}$, iteration limit $k_{\max}$, target $\tau$, and correction limit.
\Ensure Solution coefficients, recomputed residual, and convergence status.
\State Allocate Krylov workspaces; set $\mathbf c=0$, $\mathbf r_{\rm res}=\mathbf f_{\rm sys}$.
\hfill $O(k_{\max}N+k_{\max}^2)$
\State Start $B_N\Delta\mathbf c=\mathbf r_{\rm res}$ and normalize the initial
Arnoldi vector. \hfill $O(N)$
\For{$j=1,\ldots,k_{\max}$, until the stage stopping criterion is met}
\State Apply the stored right inverse.
\hfill $O((\ell+b_{\rm up}+1)N)$
\State Apply $B_N$ to the preconditioned direction. \hfill $O(MN)$
\State Orthogonalize twice against $j$ Arnoldi vectors. \hfill $O(jN)$
\State Update the small least-squares problem and residual estimate.
\hfill $O(j)$
\EndFor
\State Solve the reduced triangular system and combine stored directions.
\hfill $O(k_{\max}^2+k_{\max}N)$
\State Update $\mathbf c$; recompute $\mathbf r_{\rm res}=\mathbf f_{\rm sys}-B_N\mathbf c$.
\hfill $O(MN)$
\State Test $\|\mathbf r_{\rm res}\|\le\tau\|\mathbf f_{\rm sys}\|$; if needed and permitted,
repeat the solve with the current residual and the correction tolerance. \hfill $O(N)$
\State Reconstruct the physical coefficients and return the actual residual
and convergence status. \hfill $O(MN)$
\end{algorithmic}
\end{algorithm}
\begin{remark}[Reuse and total cost]
Algorithm~\ref{alg:gmres} reuses the contour tables, band factors, and
Krylov workspaces. For fixed node counts $M,m$, bandwidths $\ell,b_{\rm up}$,
and polynomial degrees, preparation costs $O(N)$. Including orthogonalization,
the total work is $O((1+k^2+n_{\rm chk})N)$ and storage is
$O((k+1)N+k^2)$, where $k$ counts all Arnoldi steps, including corrections,
and $n_{\rm chk}$ counts residual checks and reconstructions.
For fixed $k,n_{\rm chk}$, both costs are linear in $N$.
Neither the iteration count nor the bandwidth is assumed uniform in
the problem parameters.
\end{remark}


\section{Numerical experiments}
\label{sec:experiments}

\paragraph{Contour parameters}
All fixed-resolution contour actions use $M=160$,
$\delta=1/2$, $\theta=35^\circ$, $\kappa=8$, and
$Y=\operatorname{arsinh}(34/8)$, with $h=Y/(M-1)$.
These choices are independent of $N$.

The experiments use Julia 1.11.5 and double precision real and complex arithmetic.
Unless stated otherwise, timings use Windows x64, an AMD Ryzen 9 9950X processor with approximately
64 GB of memory, and one Julia thread and one OpenBLAS thread.
The standalone project
contains the complete recurrence~\cite{liu25} as an independent
reference; its construction does not use the contour formula.
Action tests compare
with the complete recurrence and, for selected weighted cases, with
high-precision evaluation of the defining integral. Equation tests use
analytic solutions when available. Complete-solve timings include all
method-specific construction, workspaces, factorization, iteration, and
reconstruction, with one BenchmarkTools evaluation per sample after
compilation. Input generation,
independent validation, and file output are excluded. Cached action times
are identified separately. Sample counts are stated with the results. Floating-point subnormals are flushed
to zero for both methods, while the machine-precision spectral tails
in the prescribed inputs are retained.

\subsection{Integral-action accuracy}
For a resolved spectral length proportional to $N$, prescribe input
Chebyshev coefficients $q_j$. Set
$\epsilon_{\rm mach}=\operatorname{eps}(1)$ and use
\begin{equation}\label{eq:benchmark_input}
 q_j=\upsilon_j\exp\!\left[\log(\epsilon_{\rm mach})
 \min\left\{\frac{j}{\operatorname{round}(4N/5)-1},1\right\}\right],
 \qquad 0\le j<N.
\end{equation}
The $\upsilon_j$ are random numbers in $[-1,1]$,
with $|\upsilon_0|=1$. The envelope reaches machine
precision near $4N/5$; the remaining coefficients stay at that scale.

We set $\beta=\mu$ and $\alpha=0$, and test seven orders from $0.001$ to
$0.999$ at five sizes: $N=16$, $64$, $128$, $256$, and $512$.
Both the contour action and the independent
JFP reference retain all $N+1$ output coefficients. The left panel of
Table~\ref{tab:accuracy} reports the relative Euclidean discrepancy
in the output coefficients, maximized over the five sizes at each order.

We test the exponents
$\alpha=-0.999,-0.9,-0.5,0,0.5,2$, the seven fractional orders above,
and $N=16,64,256$. The 126 coefficient inputs use
\eqref{eq:benchmark_input}, so their decay rate changes with $N$.
The weight $t^\alpha$ belongs to the basis and is never sampled in an
action. The right panel of Table~\ref{tab:accuracy} gives the largest
relative discrepancy in all $N+1$ output coefficients for each exponent;
the overall maximum is
$\WeightedMaximum$.

\begin{table}[tbp]
\centering
\caption{Integral-action accuracy at $M=160$ for dimension-dependent random spectral inputs. Left: maximum relative coefficient discrepancy over $N=16,64,128,256,512$ for each order at $\alpha=0$. Right: maximum over seven orders and $N=16,64,256$ for each endpoint exponent.}
\label{tab:accuracy}
\small\setlength{\tabcolsep}{4pt}
\begin{minipage}[t]{.48\linewidth}\centering
\textit{Random spectral inputs}\par\smallskip
\begin{tabular}{rr}
\hline
$\mu$ & Maximum discrepancy \\
\hline
0.001 & $1.92\times10^{-15}$ \\
0.01 & $7.24\times10^{-16}$ \\
0.3 & $1.67\times10^{-15}$ \\
0.5 & $2.36\times10^{-15}$ \\
0.8 & $5.93\times10^{-15}$ \\
0.99 & $1.76\times10^{-15}$ \\
0.999 & $2.54\times10^{-15}$ \\
\hline
\end{tabular}
\end{minipage}\hfill
\begin{minipage}[t]{.48\linewidth}\centering
\textit{Endpoint weights}\par\smallskip
\begin{tabular}{rr}
\hline
$\alpha$ & Maximum discrepancy \\
\hline
-0.999 & $1.58\times10^{-15}$ \\
-0.9 & $4.31\times10^{-15}$ \\
-0.5 & $3.86\times10^{-15}$ \\
0.0 & $3.48\times10^{-15}$ \\
0.5 & $5.41\times10^{-15}$ \\
2.0 & $1.11\times10^{-14}$ \\
\hline
\end{tabular}
\end{minipage}
\end{table}


For independent integral validation, eight cases with
$\alpha=-0.999,-0.5,0.5,2$, $\mu=0.01,0.8$ and $N=32$ use a fixed
exponentially decaying polynomial and five fractional-coordinate points.
The defining beta integral is evaluated at 256-bit precision by removing
the endpoint powers and applying DE quadrature on two half intervals.
Reference refinement changes the values by less than $2.6\times10^{-32}$
relative to their maximum magnitude. For the Float64 values
$v_l=2^{-\mu}t_l^{-\alpha}(\mathcal I_x^\mu f_N)(x_l)$ and their independent references $v_l^{\rm ref}$,
$\max_l\bigl(|v_l-v_l^{\rm ref}|/\max\{1,|v_l^{\rm ref}|\}\bigr)$ is
$\BetaMaximum$. High precision is used only for reference data.


\subsection{Cost of one integral action}
The timing tests use the same inputs \eqref{eq:benchmark_input}.

At $\mu=\beta=1/2$, $\alpha=0$, and $M=160$, each point in
Figure~\ref{fig:action_timing} is the median of seven BenchmarkTools
samples using a prepared plan. There are 23 dimensions through $N=131072$.
An action at that size takes $\ActionLargestSeconds$ seconds, with a fitted time slope
$\ActionSlope$ over $1024\le N\le131072$. The compact resolvent cache occupies
$\ActionCacheMB$ MB there. These timings use the integral on $[-1,1]$.

\begin{figure}[tbp]\centering
\includegraphics[width=.7\textwidth]{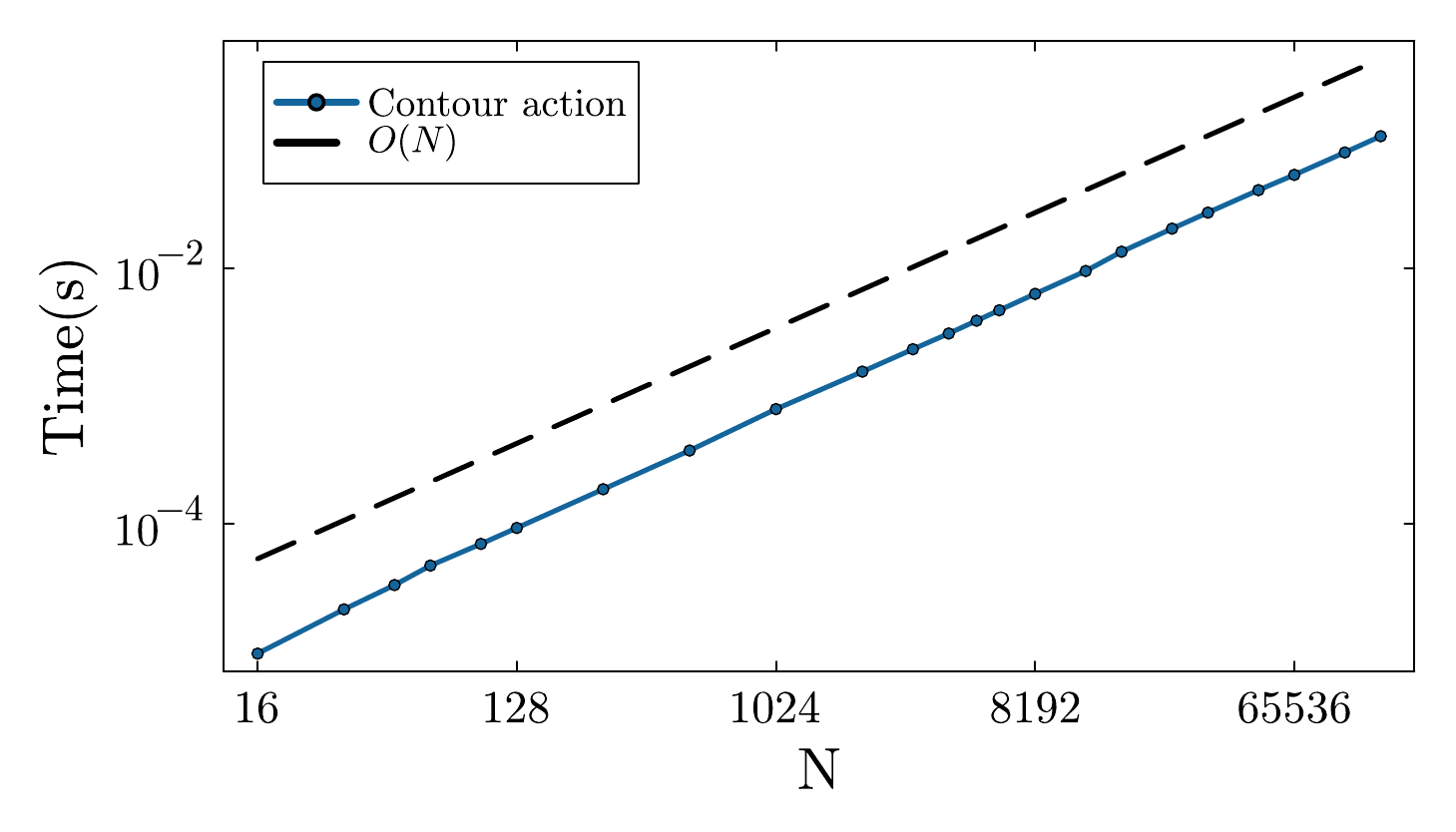}
\caption{One cached fractional-integral action on
\eqref{eq:benchmark_input}. The dashed line is an $O(N)$ reference.
No complete integration matrix is constructed.}
\label{fig:action_timing}
\end{figure}



\subsection{Abel equations}
\label{sec:abel_pc_results}
\label{sec:timing}
The Abel tests use
\[
 u(x)+2^{-\mu}\lambda(\mathcal I_x^\mu u)(x)=f(x),\qquad -1<x\le1.
\]
The complete-matrix baseline constructs the independent JFP recurrence
and uses Julia's library \texttt{UpperHessenberg} solve. Its lower
bandwidth is one, so both construction and structured solution cost
$O(N^2)$; a generic cubic factorization is not used. The contour--band
method uses the rule \eqref{eq:abel_rational_nodes}, upper bandwidth
$b_{\rm up}=32$, zero-start GMRES, and reorthogonalization, with the contour
parameters specified above. All preparation is charged to the solve.

\paragraph{Increasing spectral complexity}
The 45-point main grid uses $\alpha=0$ and
\[
 \mu\in\{0.1,0.25,0.5,0.75,0.9\},\quad
 \lambda\in\{4,100,1000\},\quad N\in\{256,1024,2048\}.
\]
It is supplemented by six problems at $\mu=0.5$, $N=128,512$;
six at $N=512$, $\lambda=100$, $\alpha=\pm0.5$,
$\mu=0.25,0.5,0.75$; and three at $N=512$, $\lambda=-1$,
$\alpha=0$, $\mu=0.25,0.5,0.75$. The right-hand sides follow
\eqref{eq:benchmark_input}. Their decay rate changes with $N$ and their
rounding-level tails are retained, so this is a cost test for increasingly
complex data, not convergence to one fixed continuous solution.

The 60 contour--band runs use \AbelMinSteps--\AbelMaxSteps{} GMRES steps
with relative stopping tolerance $10^{-14}$.
Table~\ref{tab:abel_cost_accuracy} reports the solution discrepancies
from the independent complete-matrix reference. Coefficient differences
and original equation residuals use relative Euclidean norms. The
regular-factor differences use the relative maximum norm on 2001 equally
spaced points in $z$.

\begin{table}[tbp]
\centering
\caption{Abel solution discrepancies from the independent complete-matrix reference and original relative residuals. Each entry is the maximum over the indicated parameter group.}
\label{tab:abel_cost_accuracy}
\small\setlength{\tabcolsep}{4pt}
\begin{tabular}{lrrrr}
\hline
Parameters & Cases & Coefficients & Regular factor & Residual \\
\hline
$\alpha=0,\ \lambda=4$ & 17 & $9.72\times10^{-15}$ & $1.10\times10^{-14}$ & $9.76\times10^{-15}$ \\
$\alpha=0,\ \lambda=100$ & 17 & $4.91\times10^{-14}$ & $1.12\times10^{-13}$ & $1.81\times10^{-13}$ \\
$\alpha=0,\ \lambda=1000$ & 17 & $3.16\times10^{-13}$ & $2.09\times10^{-13}$ & $2.11\times10^{-12}$ \\
$\alpha=-0.5,\ \lambda=100$ & 3 & $1.20\times10^{-14}$ & $1.22\times10^{-14}$ & $3.19\times10^{-14}$ \\
$\alpha=0.5,\ \lambda=100$ & 3 & $8.62\times10^{-15}$ & $1.13\times10^{-14}$ & $3.27\times10^{-14}$ \\
$\alpha=0,\ \lambda=-1$ & 3 & $1.92\times10^{-15}$ & $2.25\times10^{-15}$ & $1.17\times10^{-15}$ \\
\hline
\end{tabular}
\end{table}


We time 33 configurations: orders $0.1,0.5,0.9$ in the main grid and the
six additional $\mu=0.5$ configurations. Each method has three
BenchmarkTools samples, with one complete solve per sample. At
$N=2048$, $\mu=0.5$, the contour--band times for $\lambda=4,100,1000$
are $\AbelFourBandSeconds$, $\AbelHundredBandSeconds$, and $\AbelThousandBandSeconds$ seconds, compared with
$\AbelFourDirectSeconds$, $\AbelHundredDirectSeconds$, and $\AbelThousandDirectSeconds$ seconds for the complete method.
Figure~\ref{fig:abel_pc_timing}(a,b) clearly shows $O(N)$ scaling for the
contour--band method and $O(N^2)$ scaling for the complete-matrix method.

\paragraph{Timing and storage}
For an additional timing and storage test, set $\mu=\beta=1/2$,
$\alpha=0$, and $\lambda=\sqrt2$, so the equation becomes
\[
 u(x)+(\mathcal I_x^{1/2}u)(x)=f_N(x),\qquad -1<x\le1.
\]
The JFP coefficients of $f_N$ are
\[
 b_j=\exp\!\left[\log(\epsilon_{\rm mach})\frac{j}{k-1}\right],
 \quad 0\le j<k,\qquad k=\lfloor4N/5\rfloor,
\]
followed by uniform noise in $[-\epsilon_{\rm mach},\epsilon_{\rm mach}]$
from a Mersenne Twister with seed $20261005+N$.
Both methods use the same vector at each $N$.
The complete matrix uses its upper-Hessenberg structure for a direct
solve; the contour method uses $M=160$ and zero-start unpreconditioned
GMRES with tolerance $10^{-13}$ and at most 64 steps.

The timings in Figure~\ref{fig:timing_storage}(c) use
an AMD Ryzen 9 9950X, Julia 1.11.5,
and one Julia and BLAS thread, with medians over five warmed samples.
Setup includes operator preparation and workspaces; total time adds
the coefficient solve and a recomputed residual check, excluding input
generation and reconstruction. Stored representation excludes action
and solver workspaces.
At $N=16384$, setup plus solve takes $41.7$ s for the complete matrix
and $0.231$ s for the contour method, using $2048$ and $40.1$ MiB,
respectively. The contour method reaches $N=131072$ in $1.83$ s with
$321$ MiB; all contour solves take 15--16 GMRES steps.
At shared dimensions, the maximum coefficient difference is
$3.39\times10^{-13}$.

\begin{figure}[tbp]\centering
\includegraphics[width=\textwidth]{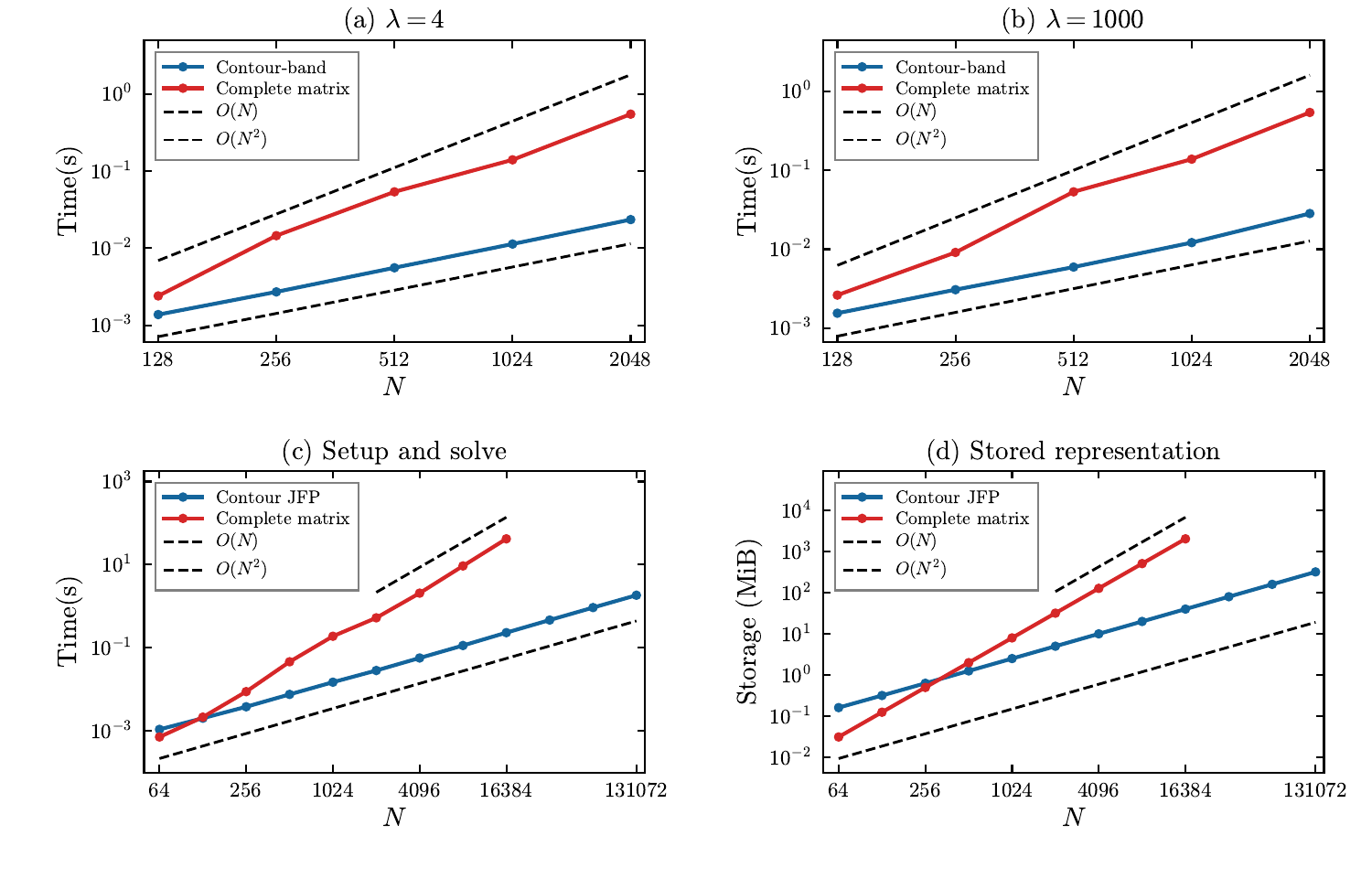}
\caption{Abel timing and storage comparisons at $\mu=0.5$, $\alpha=0$.
(a,b): complete solution cost for $\lambda=4$ and $\lambda=1000$,
respectively, with banded preconditioning and GMRES tolerance $10^{-14}$.
(c,d): setup-plus-solve time and stored representation for
$\lambda=\sqrt2$, using unpreconditioned contour GMRES at tolerance
$10^{-13}$, with the contour method tested through $N=131072$ and
the complete matrix through $N=16384$.
Black dashed lines indicate $O(N)$ and $O(N^2)$ reference rates.}
\label{fig:abel_pc_timing}
\label{fig:timing_storage}
\end{figure}


\subsection{Muti-order applications}
\label{sec:mathcomp}
Three applications from~\cite{liu25} test multiple orders and deferred
correction. The linear multi-order and BBO examples use the banded
preconditioner with $b_{\rm up}=32$; SDC uses the accurate integral action directly. Independent complete matrices
use library solvers exploiting the small lower bandwidth. The five-sample
median times in Table~\ref{tab:mathcomp_current} include construction,
solution and any reconstruction.

\begin{table}[tbp]
\centering
\caption{Muti-order applications. Times include construction and reconstruction (five samples). The multi-order and BBO contour solves use banded preconditioning; SDC uses the integral action directly. Errors are defined in the text.}
\label{tab:mathcomp_current}
\small\setlength{\tabcolsep}{4pt}
\begin{tabular}{lrrrrr}
\hline
Problem & $N$ & Direct error & Contour error & Direct (s) & Contour (s) \\
\hline
Multi-order & 30 & $1.26\times10^{-15}$ & $2.67\times10^{-15}$ & 0.00071 & 0.00098 \\
BBO & 50 & $8.88\times10^{-16}$ & $2.89\times10^{-15}$ & 0.00122 & 0.00125 \\
SDC & 10 & $1.11\times10^{-15}$ & $2.20\times10^{-14}$ & 0.00084 & 0.00148 \\
\hline
\end{tabular}
\end{table}


A superscript $\mathrm C$ marks a Caputo derivative with lower terminal $-1$.
The multi-order equation is
\begin{gather*}
 u+\sqrt{1+x}\,\mathcal I_x^{1/3}u
   +\mathcal I_x^{1/2}[(1+\cdot)^{1/3}u]=f,\\
 f(x)=(1+x)^{3/2}
 +\left[\frac{\Gamma(5/2)}{\Gamma(17/6)}
 +\frac{\Gamma(17/6)}{\Gamma(10/3)}\right](1+x)^{7/3}.
\end{gather*}
Its exact solution $(1+x)^{3/2}$ is degree nine for
$(\alpha,\beta)=(0,1/6)$. The integrals increase degree by two and three; GMRES takes
\MultiSteps{} steps at $N=30$.

For the Basset--Boussinesq--Oseen equation
\[
 v'(x)+{}^{\rm C}\!\mathcal D_x^{1/2}v(x)+v(x)=0,\qquad v(-1)=1,\quad -1\le x\le1,
\]
we solve the integral equation for $u=v'$
\[
 u+\mathcal I_x^{1/2}u+\mathcal I_x^1u=-1
\]
and reconstruct $v=1+\mathcal I_x^1u$, using
$(\alpha,\beta)=(0,1/2)$. An independent reference is
\[
 v(x)=2\Re\left\{\left(\frac12-\frac{i}{2\sqrt3}\right)
 \operatorname{erfcx}(-z_*\sqrt{1+x})\right\},\qquad z_*=\frac{-1+i\sqrt3}{2}.
\]
At $N=50$, GMRES takes \BBOSteps{} steps.
Both linear tests use zero-start GMRES, reorthogonalization and
relative tolerance $2\times10^{-14}$. Analytic errors are sampled at
2001 polynomial-coordinate points. At these small sizes the complete-matrix method remains faster.

The deferred-correction problem is
\[
 {}^{\rm C}\!\mathcal D_x^{1/2}u=u+\sqrt{1+x}-\frac{\sqrt\pi}{2}(1+x),
 \qquad u(-1)=0,
\]
whose integral form is
\[
 u(x)-(\mathcal I_x^{1/2}u)(x)
 =\frac{\sqrt\pi}{2}(1+x)-\frac23(1+x)^{3/2},
\]
with exact solution $u=\sqrt\pi(1+x)/2$. Ten mapped Chebyshev nodes
and the same piecewise-constant product-integration initializer and
correction are used for both integral representations. Both stop after
\SDCDirectSteps{} corrections at discrete residual $10^{-15}$. The absolute solution errors are reported in Table~\ref{tab:mathcomp_current}.

\begin{figure}[tbp]\centering
\includegraphics[width=\textwidth]{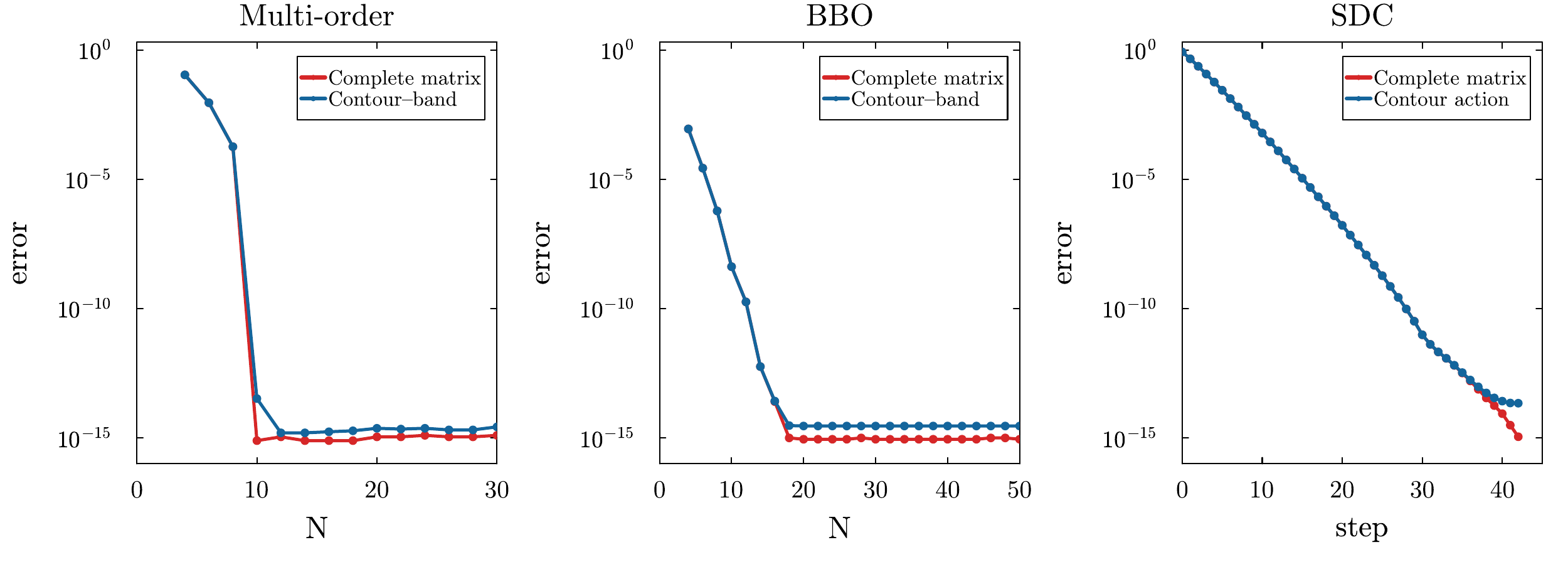}
\caption{Muti-order applications. Left and middle: sampled relative
solution errors for the multi-order and BBO equations. Right: absolute
SDC solution error. All contour actions use $M=160$.}
\label{fig:mc_sdc}
\end{figure}

\subsection{Riemann--Liouville eigenvalue problem}
\label{sec:mc_eigen}
Consider the eigenproblem from~\cite{liu25},
\[
 -\mathcal D_x^{3/2}u=\lambda_{\rm eig}u,\qquad u(-1)=u(1)=0.
\]
The equivalent integral equation is
\[
 u(x)+\lambda_{\rm eig}(\mathcal I_x^{3/2}u)(x)
 =\lambda_{\rm eig}\sqrt{\frac{1+x}{2}}(\mathcal I_x^{3/2}u)(1).
\]
Arnoldi is applied to
$u\mapsto t^{1/2}(\mathcal I_x^{3/2}u)(1)-\mathcal I_x^{3/2}u(x)$,
where $t=(1+x)/2$. Its eigenvalues are $1/\lambda_{\rm eig}$, and
$(\alpha,\beta)=(1/2,3/2)$. Both methods use Krylov dimension 40,
tolerance $10^{-13}$ and the same exponentially decaying starting vector
to find the seven largest-modulus inverse eigenvalues. At $N=500$ both
take 40 actions.

The scalar eigenvalue equation is
\[
 E_{3/2,3/2}(-2^{3/2}\lambda_{\rm eig})=0,
 \qquad E_{a,b}(z)=\sum_{j=0}^\infty\frac{z^j}{\Gamma(aj+b)}.
\]
Table~\ref{tab:eigen_mlf} lists the eigenvalues and their individual
characteristic-function values, with the conjugate pair combined in row 6.
The complete-matrix and matrix-free solves take
\EigenDirectSeconds{} and \EigenContourSeconds{} seconds, respectively
(medians of five samples).

\begin{table}[tbp]
\centering
\caption{Eigenvalues of the Riemann--Liouville problem at $N=500$, computed by the complete-matrix and contour-based matrix-free Julia implementations. The conjugate pair shares row 6. Characteristic-function values are evaluated independently at 256-bit precision using the stored Float64 eigenvalues.}
\label{tab:eigen_mlf}
\small\setlength{\tabcolsep}{4pt}
\renewcommand{\arraystretch}{1.35}
\begin{tabular}{cccl}
\hline
Realization & Index & $\lambda_{\rm eig}$ & $|E_{3/2,3/2}(-2^{3/2}\lambda_{\rm eig})|$ \\
\hline
\multirow{6}{*}{\shortstack{Complete\\matrix}} & 1 & $1.794435495663995$ & $2.35\times10^{-16}$ \\
 & 2 & $6.177290302782696$ & $1.10\times10^{-15}$ \\
 & 3 & $11.359485354307592$ & $2.56\times10^{-15}$ \\
 & 4 & $19.740438605345727$ & $3.77\times10^{-15}$ \\
 & 5 & $22.834767521667569$ & $4.16\times10^{-15}$ \\
 & 6 & $35.255579686697054\pm 7.532188956982297i$ & $5.74\times10^{-15}$ \\
\hline
\multirow{6}{*}{Matrix-free} & 1 & $1.794435495663995$ & $2.35\times10^{-16}$ \\
 & 2 & $6.177290302782539$ & $6.25\times10^{-16}$ \\
 & 3 & $11.359485354309625$ & $3.27\times10^{-16}$ \\
 & 4 & $19.740438605272438$ & $6.56\times10^{-16}$ \\
 & 5 & $22.834767521829491$ & $7.41\times10^{-16}$ \\
 & 6 & $35.255579687067758\pm 7.532188956768764i$ & $1.03\times10^{-15}$ \\
\hline
\end{tabular}
\end{table}



\subsection{Fractional Airy boundary problem}
\label{sec:mc_airy}
\label{sec:airy_rational}

Consider the Riemann--Liouville boundary problem
\begin{equation}\label{eq:mc_airy}
 \epsilon i^{3/2}\mathcal D_x^{3/2}u-xu=0,\quad -1<x<1,
 \qquad u(-1)=0,\quad u(1)=1,
\end{equation}
with the principal value of $i^{3/2}$. Writing
$u=\mathcal I_x^{3/2}v+a\sqrt{(1+x)/2}$, where
$v=\mathcal D_x^{3/2}u$, gives the integral equation
\[
 \epsilon i^{3/2}v(x)-x(\mathcal I_x^{3/2}v)(x)
 -ax\sqrt{\frac{1+x}{2}}=0,
\]
with $a+(\mathcal I_x^{3/2}v)(1)=1$.
The augmented preconditioning system is
\begin{equation}\label{eq:airy_band_system}
{\renewcommand{\arraystretch}{1.25}
 \left(\begin{array}{c|c}
  2^{3/2}\mathbf1_N^\T R_{3/2,1/2}^{(N)} & 1\\
  \hline
  P_N & -\mathbf x
 \end{array}\right)
 \begin{pmatrix}\mathbf v\\a\end{pmatrix}
 =\begin{pmatrix}1\\\mathbf0\end{pmatrix}.
}
\end{equation}
Here $P_N$ is the banded interior block, $\mathbf1_N$ is the all-ones
vector, and $\mathbf v,\mathbf x$ represent the auxiliary unknown and
$x\sqrt{(1+x)/2}$ in the JFP basis with $(\alpha,\beta)=(1/2,1/2)$.
The first row imposes $u(1)=1$; eliminating $a$ gives a rank-one
correction to the banded solve. The Gamma recurrence reduces the rational
approximation to order $1/2$, using 32 trapezoidal nodes on $[-5,15]$
and an exact real resolvent. The lower bandwidth is five.

We use $\epsilon=10^{-4},10^{-5},10^{-6},10^{-7},10^{-8}$ with upper
bandwidths $32,32,32,64,128$, respectively. Zero-start GMRES uses relative
tolerance $10^{-13}$ for the first three parameters and $10^{-11}$ for
the last two, followed by at most three original-residual corrections
at relative tolerance $10^{-3}$. Reported steps include all corrections.

\begin{figure}[tbp]\centering
\subfloat[$\epsilon=10^{-4}$]{\includegraphics[width=.32\textwidth]{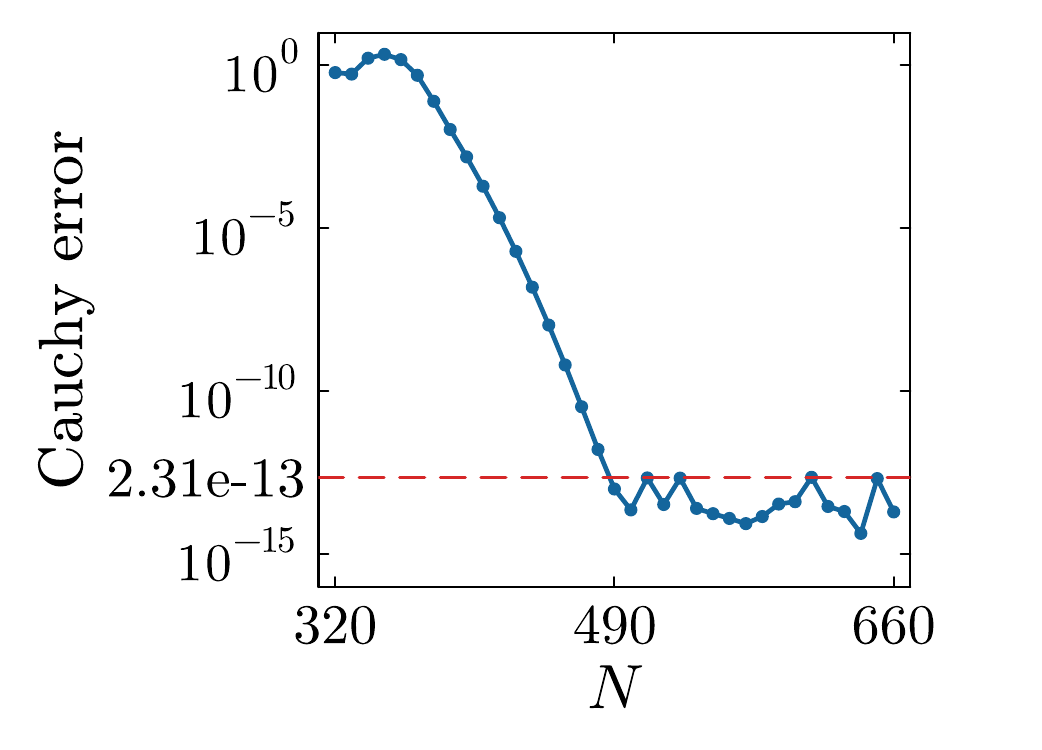}}\hfill
\subfloat[$\epsilon=10^{-5}$]{\includegraphics[width=.32\textwidth]{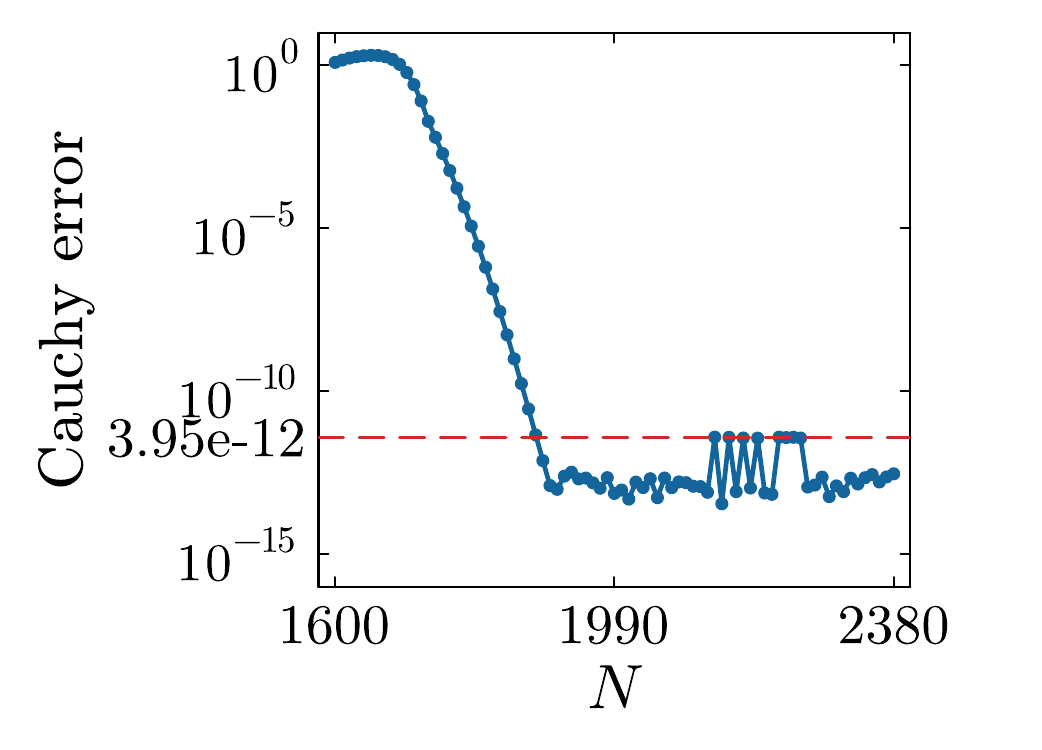}}\hfill
\subfloat[$\epsilon=10^{-6}$]{\includegraphics[width=.32\textwidth]{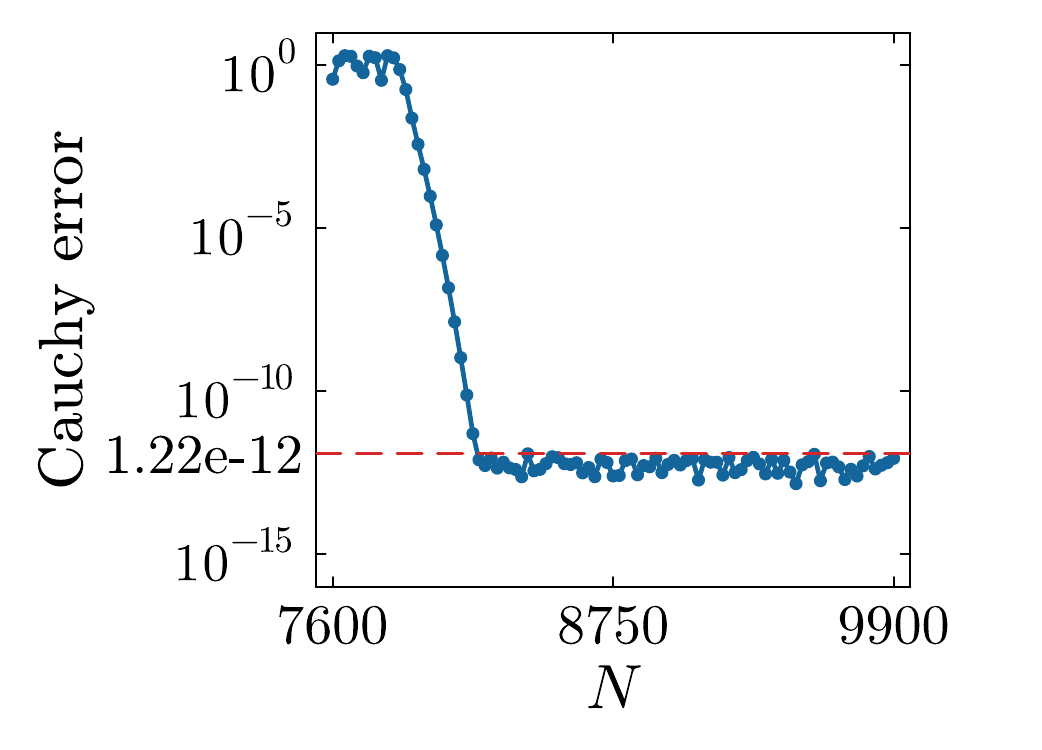}}\\[4pt]
\subfloat[$\epsilon=10^{-7}$]{\includegraphics[width=.32\textwidth]{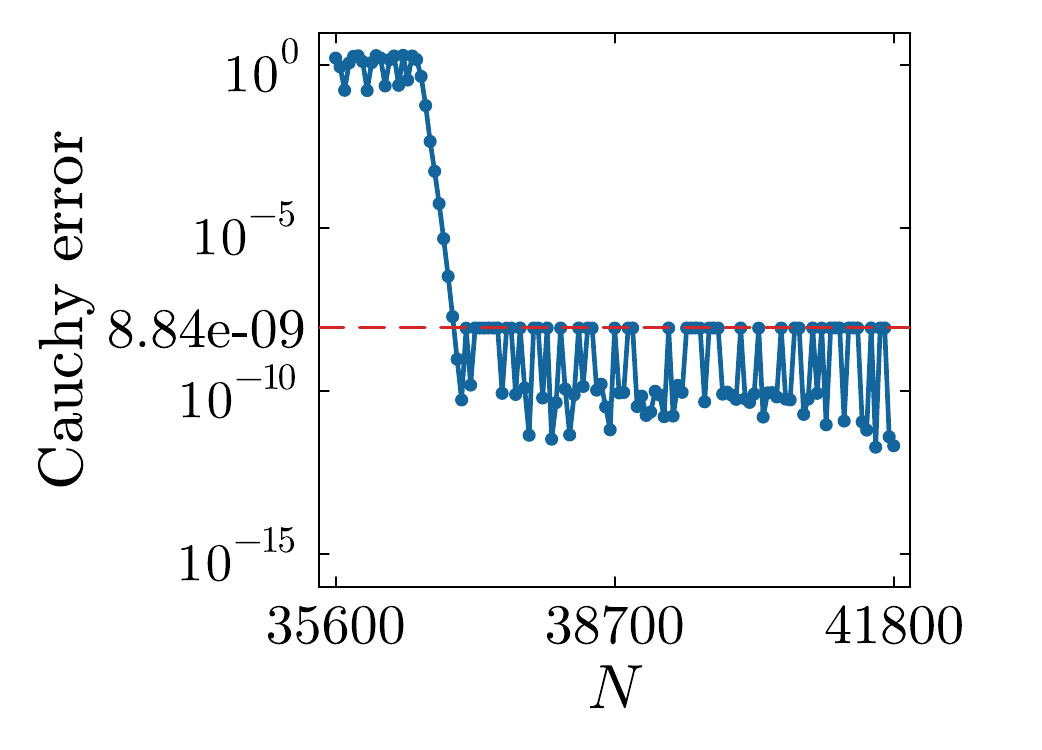}}\hfill
\subfloat[$\epsilon=10^{-8}$]{\includegraphics[width=.32\textwidth]{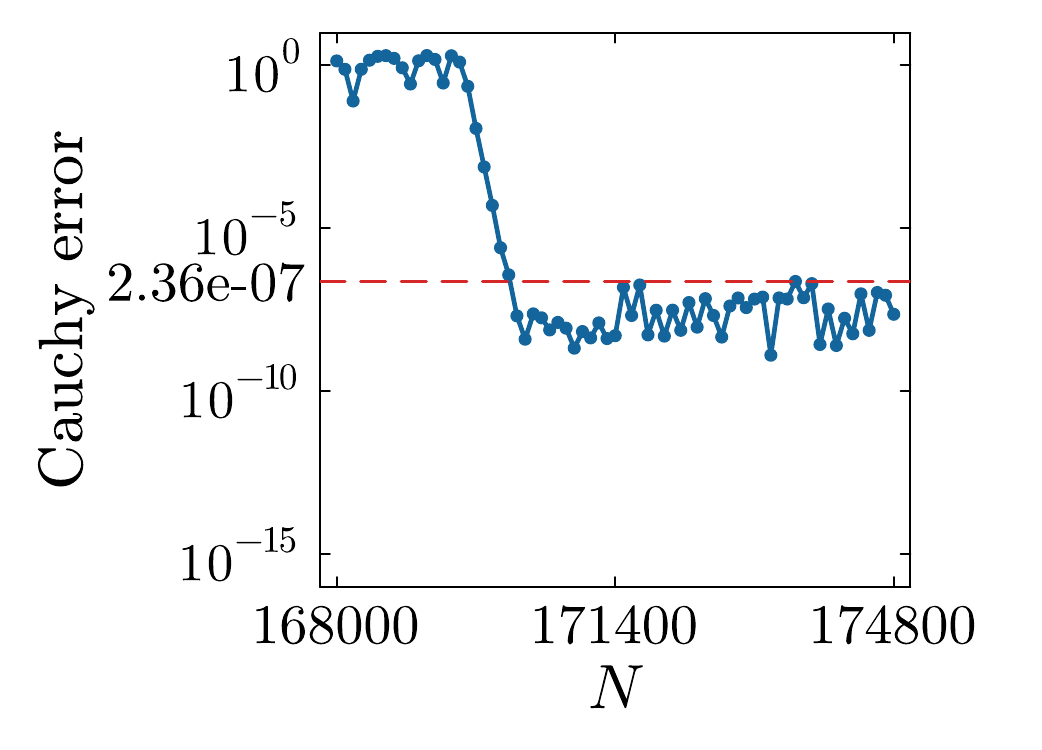}}\hfill
\subfloat[$\epsilon=10^{-4}$, $N=2000$]{\includegraphics[width=.32\textwidth]{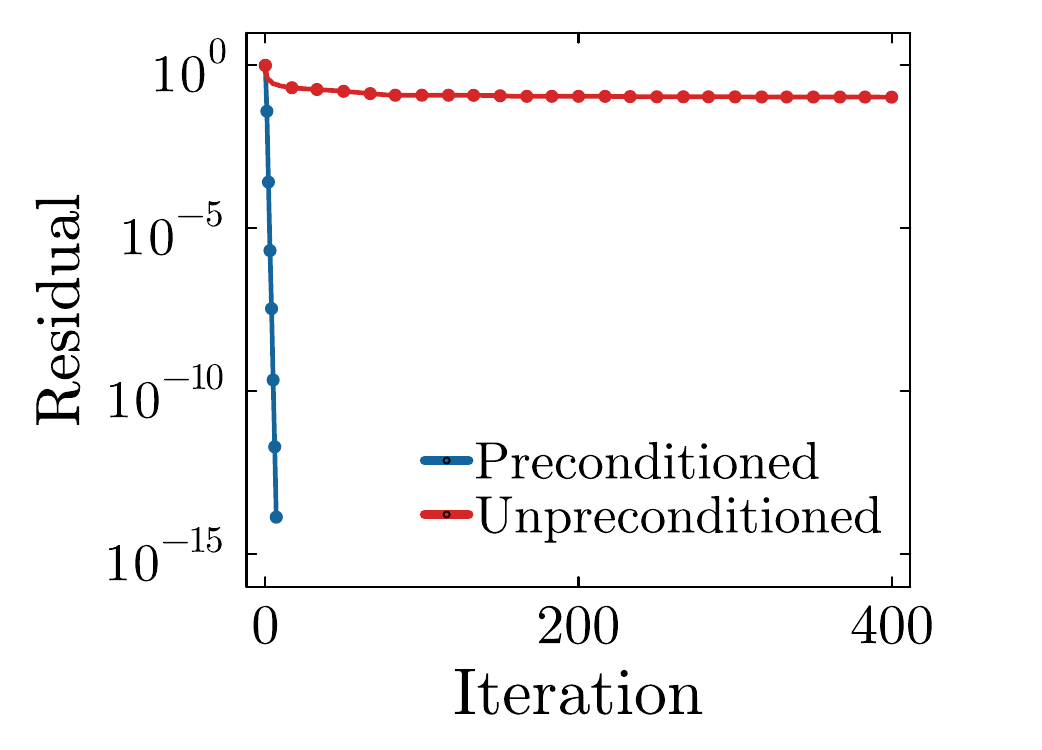}}
\caption{Airy results. (a)--(e): relative sampled maximum-norm differences
$E_N$ between solutions at $N$ and $N+\Delta N$, with
$\Delta N=20,20,100,200,200$, respectively, normalized by the latter.
(f): initial GMRES residual histories with banded preconditioning
(blue circles) and without (red squares), at tolerance $10^{-13}$.}
\label{fig:airy_cauchy}
\end{figure}

Table~\ref{tab:airy_large} summarizes the final scan resolutions.
For the reconstructed coefficients $u_j$, $j_*$ is the first zero-based
index such that $|u_j|\le\epsilon_{\rm mach}\max_l|u_l|$ for all $j\ge j_*$.

\begin{table}[tbp]\centering
\caption{Final Airy scan resolutions. Times include preparation, all residual corrections, and reconstruction (one BenchmarkTools sample).}
\label{tab:airy_large}
\small\setlength{\tabcolsep}{5pt}\begin{tabular}{rrrrrrr}\hline
$\epsilon$ & $N$ & $b_{\rm up}$ & Total steps & Original residual & Time (s) & $j_*$\\\hline
$10^{-4}$ & 680 & 32 & 7 & $7.32\times10^{-14}$ & 0.016 & 514 \\
$10^{-5}$ & 2400 & 32 & 13 & $1.63\times10^{-14}$ & 0.083 & 1942 \\
$10^{-6}$ & 10000 & 32 & 36 & $3.11\times10^{-13}$ & 0.736 & 8280 \\
$10^{-7}$ & 42000 & 64 & 36 & $4.51\times10^{-13}$ & 3.360 & 37207 \\
$10^{-8}$ & 175000 & 128 & 194 & $5.51\times10^{-12}$ & 66.015 & 170663 \\
\hline\end{tabular}\end{table}


At $\epsilon=10^{-4}$, $N=2000$, preconditioning reduces the GMRES
residual estimate to $1.39\times10^{-14}$ in seven steps, compared with
$0.106$ after 400 steps without preconditioning
(Figure~\ref{fig:airy_cauchy}(f)).

At $\epsilon=10^{-8}$, $N=\SweepEightN$, three residual corrections
reduce the original residual from $\SweepEightInitialResidual$ to
$\SweepEightThirdResidual$, with \SweepEightSteps{} total steps.
The physical Cauchy differences decay rapidly before fluctuating between
orders $10^{-9}$ and $10^{-7}$ (Figure~\ref{fig:airy_cauchy}(e)).


\FloatBarrier
\section{Conclusion}
We have presented a fast contour method for Riemann--Liouville fractional
integration in a Chebyshev-type JFP basis. The Gamma-ratio representation
reduces the action to banded shifted solves. The quadrature error decays
superalgebraically, and the node count required for a fixed tolerance is
independent of the basis dimension. This gives $O(N)$ action cost and
storage. Direct construction of selected rational diagonals also yields
banded right preconditioners with linear preparation and application
costs for fixed bandwidth.

Neumann-series estimates establish supergeometric convergence of exact
operator GMRES for Abel equations and finite fractional integral sums
with bounded variable coefficients. Numerical experiments demonstrate
accurate integral actions, spectral convergence for the analytic examples,
and reduced solution costs across integral equations, eigenvalue problems,
and the fractional Airy boundary problem. Further work will address the
convergence of the preconditioned coefficient systems and extensions to
general Jacobi parameters.

\section*{Code Availability}
The Julia implementation, single-case examples, and locked dependency
environment are available at \url{https://github.com/xiaolin9907/FastJFP}.

\end{document}